\documentclass{amsart}
\usepackage{amsmath,amsthm,amssymb}
\usepackage{lipsum} 
\usepackage[utf8]{inputenc}
\usepackage{esint}
\usepackage{amsfonts} 
\usepackage{xcolor}
\usepackage{parskip}
\usepackage{tikz}
\usepackage{soul}
\usepackage[colorlinks=true]{hyperref}
\usepackage{bbold}
\usepackage{enumerate}
\usepackage{calc}
\usepackage{todonotes}
\usepackage[shortcuts]{extdash}
\usepackage[style=numeric,giveninits,doi=true,isbn=false,url=false,eprint=false,backend=bibtex]{biblatex}
\makeatletter
\numberwithin{equation}{section}
\theoremstyle{plain}
        \newtheorem{theorem}{Theorem}[section]

        \newtheorem{assumption}{Assumption}[section]
        \newtheorem{proposition}[theorem]{Proposition}
        \newtheorem{lemma}[theorem]{Lemma}
        \newtheorem{corollary}[theorem]{Corollary} 
        
        \newtheorem{definition}[theorem]{Definition} 
        \newtheorem{remark}[theorem]{Remark}  
        \newtheorem{example}[theorem]{Example}
         
        \newtheorem*{claim*}{Claim}

        \newtheorem*{fact*}{Fact} 
        
\newtheorem*{theorem*}{Theorem}
\newtheorem*{definition*}{Definition}
\newtheorem*{proposition*}{Proposition}
\usepackage{graphicx}

\renewcommand\div{\text{div\,}}
\newcommand \loc {\text{loc}}

\newcommand{\R}{\mathbb{R}}

\newcommand{\Q}{\mathbb{Q}}

\newcommand{\N}{\mathbb{N}}

\newcommand{\f}{f^{\alpha,\beta,\gamma,\delta}_t}

\newcommand{\fr}{f^{\alpha,\beta,\gamma,\delta}_r}

\newcommand{\Ur}{U^{\alpha,\beta,\gamma,\delta}_r}

\newcommand{\cF}{\mathcal{F}}

\newcommand{\cM}{\mathcal{M}}
\newcommand{\cL}{\mathcal{L}}
\newcommand{\cP}{\mathcal{P}}
\newcommand{\cU}{\mathcal{U}}
\newcommand{\cR}{\mathcal{R}}

\newcommand{\cD}{\mathcal{D}}
\newcommand{\cE}{\mathcal{E}}
\newcommand{\cH}{\mathcal{H}}

\newcommand{\TCRE}{\mathcal{TCRE}}
\newcommand{\Do}{\R^{2d}}
\newcommand{\G}{\R^{4d}\times S^{d-1}}
\newcommand{\sgn}{\text{sgn}}

\renewcommand{\d}{\partial} 
\newcommand{\dd}{{\,\rm d}}

 \newcommand{\jb}[1]{\langle #1\rangle}

 \newcommand{\xs}{x_*}
 \newcommand{\vs}{v_*}
\newcommand{\vp}{v'}
\newcommand{\vsp}{v'_*}
\newcommand{\fS}{f_*}
\newcommand{\fP}{f'}
\newcommand{\fSP}{f'_*}

\title{A variational approach to a fuzzy Boltzmann equation}

\author[M.~Erbar]{Matthias Erbar}
\author[Z.~He]{Zihui He}
\address[M.~Erbar and Z.~He]
{Fakult\"at f\"ur Mathematik, Universit\"at Bielefeld, Postfach 100131, 33501 Bielefeld, Germany}
\email{erbar@math.uni-bielefeld.de, zihui.he@uni-bielefeld.de}

\hypersetup{final}

\begin{document}

\begin{abstract}
We study a fuzzy Boltzmann equation, where particles interact via delocalised collisions, in contrast to classical Boltzmann equations.
We discuss the existence and uniqueness of solutions and provide a natural variational characterisation by casting the fuzzy Boltzmann equation into the framework of GENERIC systems (General Equations for Non-Equilibrium Reversible-Irreversible Coupling).
\end{abstract}
\maketitle

\section{Introduction}

In this paper we are interested in the variational characterisation of PDEs that combine dissipative as well as conservative, Hamiltonian effects. Purely dissipative evolution equations often admit a variational description as a gradient flow, or steepest descent,  $\partial_tx=-\mathsf M D\mathsf S$ of an entropy functional $\mathsf S$ for an appropriate notion of geometry on the space the solution evolves in that is encoded by a symmetric non-negative operator $\mathsf M$. For instance, vast classes of diffusive PDEs can be seen as gradient flows for the Wasserstein geometry on the space of probability measures. This interpretation has been made rigorous and exploited in many cases, see \cite{AGS08} for an overview. Also other geometries adapted to different evolution mechanisms have been considered more recently. Identifying the variational structure of a PDE as a gradient is a powerful tool for its analysis, providing general methods for proving well-posedness \cite{AGS08} and studying the trend to equilibrium \cite{CMV03}, a robust framework for analysing singular limits \cite{SS04, AMPSV12}, or giving rise to structure-preserving numerical schemes.

In the context of evolution equations stemming from statistical mechanics, gradient flows correspond to systems close to thermodynamic equilibrium. 
For systems far from equilibrium, one expects a combination of conservative and dissipative effects in the dynamics. The framework GENERIC (General Equations for Non-Equilibrium Reversible Irreversible Coupling, see \cite{Ott05}) provides a flexible tool to derive thermodynamically consistent evolution equations for such systems. The key building blocks are an energy functional $\mathsf E$, an entropy functional $\mathsf S$ as well as a Poisson operator $\mathsf L$ and a dissipative operator $\mathsf M$ , subject to certain properties and relations, and the evolution takes the form $\partial_t \mathsf z=\mathsf L \mathsf d\mathsf E+\mathsf M \mathsf d\mathsf S$. Also the GENERIC interpretation of an evolution equation encodes a variational structure, as has been pointed out in \cite{DPZ13} and as we shell describe in detail below. This structure is tightly linked to the large deviations behaviour of microscopic stochastic particle systems underlying the macroscopic evolution equation in question.

An emblematic example of such out-of equilibrium dynamics is the Boltzmann equation describing the evolution of the 1-particle distribution of a rarefied gas. The mathematical analysis of this equation is extremely challenging and many fundamental questions regarding e.g. the Cauchy problem or the rigorous derivation from particle dynamics have not yet been answered completely. Having a variational structure for this equation at disposal might facilitate further progress on some of these questions, which is our main motivation here. In order to make progress in this direction, we consider in this paper as a toy model a \emph{fuzzy} Boltzmann equation in which particle collisions are delocalised. We give a rigorous variational description of it based on its GENERIC structure by characterising solutions as the minimisers of a natural functional involving the entropy and suitable dissipation terms. The delocalisation takes away some of the immense technical difficulties associated with the Boltzmann equation while preserving most of its features. We build on earlier work in \cite{Erb23} of the first author, where a gradient flow characterisation of the spatially homogeneous Boltzmann equation has been given.

In the remaining part of this introduction we will first introduce the fuzzy Boltzmann equation and discuss its well-posedness. Then we will describe in more detail the abstract GENERIC framework and its variational structure. Finally, we present our results giving a variational characterisation of the fuzzy Boltzmann equation.

\subsection{The (fuzzy) Boltzmann equation}

The classical Boltzmann equation is given by
 \begin{equation}
\label{IMHBE}
\left\{
\begin{aligned}
&\d_tf+v\cdot\nabla_x f=Q(f), \\
&Q(f)=\int_{\R^d\times S^{d-1}}
\big(f' f'_*-ff_*\big)B(v-v_*,\omega)\dd v_*\dd \omega,
\end{aligned}
\right.
\end{equation} 
where the unknown $f:[0,T]\times \R^{2d}\to \R_+$ with $d\geq 3$ corresponds to the density of particles at time $t$ with position $x\in\R^d$, and velocity $v\in\R^d$. 
The linear transport term $v\cdot\nabla_x f$ describes the advection of the density by the velocity of the particles, while the collsion term $Q(f)$ describes the interaction between particles via elastic collisions. Here $v,\vs$ and $\vp,\vsp$ denote the
velocities before and after a collision respectively which are related
according to
\begin{align}\label{eq:pre-post}
  \vp = v - \jb{v-\vs,\omega}\omega\;,\quad \vsp = \vs +
  \jb{v-\vs,\omega}\omega\;,\quad \omega\in S^{d-1}\;,
\end{align}
and we use the notation
$f=f(x,v)$, $f_*=f(x,\vs)$, $f'=f(x,\vp)$, $f_*'=f(x,\vsp)$.
The collision kernel $B$ encoded the specific details of the underlying collision process. We refer to \cite{Vil02} for an in-depth review on the Boltzmann equation. One major difficulty in its analysis stems from the observation that while the natural a priori bounds on solutions include the conservation of mass, i.e.~$L^1$-norm, of $f$, the collision term typically requires $L^2$-integrability of $f$ in the spatial variable $x$ to be meaningful. For this reason, the general existence theory requires the notion of renormalised solution, see \cite{DL89b,AV02,GS10}.
 
We consider in this paper the following \emph{fuzzy} Boltzmann equation
\begin{equation}
\label{FBE}
\d_t f+v\cdot \nabla_x f=Q_{\sf fuz}(f)\;,
\end{equation}
where particles with states $(x,v)$ and $(x_*,v_*)$ interact via \emph{delocalised collisions} obtaining new states $(x,v')$ and $(x_*,v_*')$ with $v',v_*'$ given by \eqref{eq:pre-post}. More precisely, the 
modified collision term $Q_{\sf fuz}(f)$ is defined by
\begin{equation*}
Q_{\sf fuz}(f)=\int_{\R^{2d}\times S^{d-1}}
\big(f'f'_*-ff_*\big)B(v-v_*,w)k(x-x_*)\dd x_*\dd v_*\dd \omega\;,
\end{equation*}
where we now set
\begin{align}\label{eq:shorthand}
f=f(x,v),\quad f_*=f(x_*,v_*),\quad f'=f(x,v'),\quad f'_*=f(x_*,v'_*)\;,
\end{align}
and $k:\R^d\to\R^+$ is a kernel modulating the spatial interaction. We assume that the collision kernel  $B:\R^{d}\times S^{d-1}\to \R_+$ is bounded away from zero and grows at most like $\langle v-v_*\rangle^{\mu}$, for some $\mu\in(-\infty,1]$, where $\langle v \rangle:=\sqrt{1+|v|^2}$ denotes the Japanese bracket. For concreteness, we assume that $k$ is proportional to $\exp(-\gamma \jb{y})$ with $\gamma>0$. The detailed assumption can be found in Assumption~\ref{CK}.  

 Note that the fuzzy collision term has the a priori bound 
$\|Q_{\sf fuz}(f)\|_{L^1}\lesssim\|B\|_{L^\infty}\|f\|_{L^1}^2$.
As a consequence, the fuzzy Boltzmann equation \eqref{FBE} is structurally closer to the spatially {\it homogeneous} Boltzmann equation (the distribution of particles independent of position $x$)
\begin{equation}
\label{HBE}
\d_tf(v)=Q(f)(v), \quad f:[0,\infty)\times \R^d\to\R_+\;,
\end{equation}
whose well-posedness problems have been studied in an $L^1$-framework e.g. in \cite{Ark72,Ark72b,MW99,Wen99,Lu99}. 

We will first address the solvability of the fuzzy Boltzmann equation \eqref{FBE}  (see Theorem \ref{thm:1} below). We show in particular that for initial data $f_0\in L^1(\R^{2d};\R_+)$ with $(\jb{x}^2+\jb{v}^{2+\mu} )f_0,  f_0\log f_0 \in L^1(\R^{2d})$ there exist a unique weak solution $(f_t)\in C^1\big([0,\infty),L^1(\R^{2d})\big)$. Moreover, conservation of mass, momentum and energy holds, i.e.
\begin{equation*}
\int_{\R^{2d}}(1,v,|v|^2)f_t = \int_{\R^{2d}}(1,v,|v|^2)f_0\quad \forall t\geq 0\;.
\end{equation*}
Hence, we can restrict attention to probability densities $f$ in the sequel. Furthermore, we show that the \emph{entropy identity} holds: Denoting by $\cH(f)=\int_{\R^{2d}}f\log f $ the Boltzmann-Shannon entropy of the probability density $f$, we have
\begin{equation*}
    \cH(f_t)-\cH(f_s)=-\int_s^t \cD(f_r)\dd r\le0\quad \forall s\leq t\;,
\end{equation*}
where 
\begin{equation}
\label{eq:def-dissipation}
\cD(f):=\frac14\int_{\R^{4d}\times S^{d-1}}\big(f'f'_*-ff_*\big)\big(\log(f'f'_*)-\log(ff_*)\big)Bk \dd \sigma\;,
\end{equation}
is the entropy dissipation functional where $\dd\sigma=\dd x\dd x_*\dd v\dd v_*\dd\omega$ denotes the Hausdorff measure on $\R^{4d}\times S^{d-1}$.

To obtain the previous results, we follow the approach of Arkeryd \cite{Ark72} concerning the solvability and to establish an entropy inequality $\cH(f_t)-\cH(f_0)\leq -\int_s^t\cD(f_r)\dd r$. The entropy identity will then be a consequence of a refined analysis as part of the variational characterisation of solutions in Theorem \ref{thm:main-intro} below.

In subsequent work, we consider the limit of vanishing fuzziness, where the spatial kernel $k$ localises to a Dirac measure (i.e. taking $\alpha\to\infty$ in the definition of $k$). We show that the unique solution to the fuzzy Boltzmann equation converges up to subsequences to a renormalised solution of the classical Boltzmann equation \eqref{IMHBE}.

Let us note that a delocalised collision term has first been proposed by Morrey \cite{Mor55} in the context of the Maxwell-Boltzmann equation. Various other models feature a similar non-local collision term such as the Povzner equation \cite{Pov62}, where the collision kernel $B$ and the post-collisional velocities depend on the positions $x$ and $x_*$, or the Enskog equation describing a dense gas and taking the size of particles into account, see \cite{Ark86,AC90}.

\subsection{GENERIC structures}

The GENERIC formalism has been introduced in \cite{GO97a,GO97b}, we refer to the book \cite{Ott05} for an extensive treatment. In its abstract form, a GENERIC system consists of the building blocks $(\mathsf Z, \mathsf E, \mathsf S, \mathsf L, \mathsf M)$ and describes the evolution of an unknown $\mathsf z$ in the state space $\mathsf Z$ via the equation
\begin{equation}\label{eq:generic}
\partial_t \mathsf z = \mathsf L \mathsf d\mathsf E+\mathsf M \mathsf d\mathsf S\;.
\end{equation}
The functionals $\mathsf E,\mathsf S:\mathsf Z\to \R$ are interpreted as energy and entropy and $\mathsf d\mathsf E$, $\mathsf d\mathsf S$ are appropriate differentials. Moreover, one assumes that:
\begin{itemize}
\item $\mathsf L=\mathsf L(\mathsf z)$ for each $\mathsf z\in \mathsf Z$ is an antisymmetric operator mapping cotangent to tangent vectors and satisfying the Jacobi identity
\begin{equation*}\label{eq:jacobi}
\{\{\mathsf G_1,\mathsf G_2\}_{\mathsf L},\mathsf G_3\}_{\mathsf L}+\{\{\mathsf G_2,\mathsf G_3\}_{\mathsf L},\mathsf G_1\}_{\mathsf L}+\{\{\mathsf G_3,\mathsf G_1\}_{\mathsf L},\mathsf G_2\}_{\mathsf L}=0\;,
\end{equation*}
for all functions $\mathsf G_i:\mathsf Z\to \R$, $i=1,2,3$, where the Poisson bracket $\{\cdot,\cdot\}_{\mathsf L}$ is defined by
\[\{\mathsf F,\mathsf G\}_{\mathsf L}=\mathsf d\mathsf F\cdot\mathsf L \mathsf d\mathsf G\;.\]
Here, $\cdot$ signifies the duality pairing between cotangent and tangent space of $\mathsf Z$ and the asymmetry of $\mathsf L$ reads as $\mathsf v\cdot \mathsf L\mathsf w= -\mathsf w^*\cdot \mathsf L\mathsf v^*$, where $\mathsf v^*,\mathsf w^*$ refer to the dual vectors of $\mathsf v,\mathsf w$;
\item $\mathsf M=\mathsf M(\mathsf z)$ for each $\mathsf z\in \mathsf Z$ is a symmetric and positive semi-definite operator mapping cotangent to tangent vectors in the sense that
\[\mathsf v\cdot \mathsf M\mathsf w=\mathsf w^*\cdot \mathsf M\mathsf v^*\;,\quad \mathsf v^*\cdot \mathsf M\mathsf v\geq 0\quad \forall \mathsf v,\mathsf w\;;\]
\item the following \emph{degeneracy conditions} are satisfied
\begin{equation}\label{eq:degen}
\mathsf L \mathsf d \mathsf S=0\;,\quad \mathsf M \mathsf d\mathsf E = 0\;.
\end{equation}
\end{itemize}
The last property ensures that along any solution to \eqref{eq:generic} the energy $\mathsf E$ is conserved and the entropy $\mathsf S$ in non-decreasing:
\begin{equation*}
\frac{\dd}{\dd t} \mathsf E(\mathsf z_t) = 0\;,\quad \frac{\dd}{\dd t} \mathsf S(\mathsf z_t) = \mathsf d \mathsf S\cdot \mathsf M \mathsf d\mathsf S \geq 0\;.
\end{equation*}
The term $\mathsf L\mathsf d \mathsf E$ describes the conservative Hamiltonian part of the evolution, while the second term $\mathsf M\mathsf d\mathsf S$ corresponds to the dissipative part. Choosing in particular $\mathsf E=0$ the equation \eqref{eq:generic} becomes $
\partial_t \mathsf z = \mathsf M\mathsf d\mathsf S=\nabla \mathsf S$, describing the (upward) \emph{gradient flow} of the functional $\mathsf S$ in the geometry determined by $\mathsf M$ by considering the Riemannian metric that is given formally on the tangent and cotangent space by
\[
(\mathsf x,\mathsf y)_{\mathsf M^{-1}}=\mathsf x\cdot \mathsf M^{-1} \mathsf y\;,\quad
(\mathsf v,\mathsf w)_{\mathsf M}=\mathsf v\cdot \mathsf M \mathsf w\;.
\]

As pointed out in \cite{DPZ13}, the GENERIC equation \eqref{eq:generic} can formally be characterised variationally as follows: for a curve $\mathsf z:[0,T]\to\mathsf Z$ define the functional
\begin{equation}\label{eq:J-quad}
J(\mathsf z) = -\mathsf S(\mathsf z_T)+\mathsf S(\mathsf z_T) +\frac12 \int_0^T \|\partial_t \mathsf z -\mathsf L\mathsf d\mathsf E\|_{\mathsf M^{-1}}^2 + \|\mathsf d\mathsf S\|_{\mathsf M}^2\; \dd t\;.
\end{equation}
Then $J(\mathsf z)\geq 0$ and a curve $\mathsf z:[0,T]\to\mathsf Z$ is a solution to \eqref{eq:generic} if and only if $J(\mathsf z)=0$. This is a consequence of the anti-symmetry of $\mathsf L$ and the degeneracy conditions \eqref{eq:degen} by noting that 
\[
\|\partial_t\mathsf z -\mathsf L\mathsf d \mathsf E - \mathsf M \mathsf d \mathsf S \|_{\mathsf M^{-1}}^2 
= \|\partial_t \mathsf z -\mathsf L\mathsf d\mathsf E\|_{\mathsf M^{-1}}^2 + \|\mathsf d\mathsf S\|_{\mathsf M}^2
- 2(\partial_t \mathsf z, \mathsf M\mathsf d \mathsf S)_{\mathsf M^{-1}}\;.
\]
In the case of a gradient flow ($\mathsf E=0$), this characterisation is well-known and corresponds to De Giorgi's characterisation of gradient flows as curves of maximal slope (see \cite{AGS08}). 

The framework above can be generalised by introducing a pair of dual dissipation potentials $\mathsf R:T\mathsf Z\to[0,\infty)$ and $\mathsf R^*:T^*\mathsf Z\to [0,\infty)$ defined on the (co)-tangent spaces and that are convex conjugate to each other and to consider the functional
\begin{equation}\label{eq:J-gen}
J(\mathsf z) = -\mathsf S(\mathsf z_T)+\mathsf S(\mathsf z_T) + \int_0^T \mathsf R(\partial_t \mathsf z -\mathsf L\mathsf d\mathsf E) + \mathsf R^*(\mathsf d\mathsf S)\; \dd t\;.
\end{equation}
Now, again $J(\mathsf z)\geq 0$ and a the curves satisfying $J(\mathsf z)=0$ are the solutions to 
\begin{equation}\label{eq:generic-gen}
\partial_t\mathsf z = \mathsf L\mathsf d\mathsf E + \partial \mathsf R^*(\mathsf d\mathsf S)\;,
\end{equation}
where $\partial \mathsf R^*(\mathsf d\mathsf S)$ replaces the linear term $\mathsf M\mathsf d\mathsf S$. This is a consequence of the duality estimate
\[
\mathsf x\cdot \mathsf v \leq \mathsf R(\mathsf x)+\mathsf R^*(\mathsf v)\quad \text{with equality iff } \mathsf x=\partial \mathsf R^*(\mathsf v)\;.
\]
One compelling motivation for considering such generalised variational structures is their connection to the large deviation behaviour of microscopic particle systems underlying the macroscopic evolution \eqref{eq:generic-gen} in concrete situations. Namely, it has been realised that the functional $J$ appears as the rate functional for path-level large deviations of the underlying microscopic system, see \cite{ADPZ11,MPR14} in the context of gradient flows and \cite{DPZ13,KLMP18,PRS24} for connections with GENERIC.
Here, the quadratic structure \eqref{eq:J-quad} typically arises for diffusive systems, while generalised structures \eqref{eq:J-gen} are connected e.g. for systems with jumps. Also generalised variational structures appear naturally in certain singular limits, see e.g.~\cite{LMPR17}.

The variational characterisation \eqref{eq:J-gen} has been made rigorous for instance in \cite{DPZ13} for kinetic Fokker-Planck equations or \cite{BBB17} for linear Boltzmann equations. In an abstract setting, it has been used in \cite{JST22} to develop approximation schemes for solutions similar to the minimising movement scheme for gradient flows.

\subsection{Variational structure of the fuzzy Boltzmann equation}

Our main goal in this paper is to exploit the GENERIC structure and make the above variational characterisation rigorous in the case of the fuzzy Boltzmann equation \eqref{FBE}. In order to cast this equation formally into the framework above, we consider as state space $\mathsf Z$ the space of probability densities on $\R^{2d}$ and the define the energy and entropy functionals as
\begin{align*}
\mathsf E(f)&=\frac12\int_{\R^{2d}}|v|^2f \;,\quad \mathsf S(f)=-\cH(f) = -\int_{\R^{2d}} f\log f\;.
\end{align*}
For each $f$ we define operators $\mathsf M(f)$ and $\mathsf L(f)$ acting on sufficiently smooth and decaying functions $g$ via 
\begin{equation}\label{eq:LM}
\begin{aligned}
\mathsf L(f)g &=-\div(f J\nabla g),\quad J=\begin{pmatrix}
0 & {\sf id}\\
-{\sf id} & 0
\end{pmatrix}\;,\\
\mathsf M(f)g&=-\overline\nabla\cdot (Bk\Lambda (f)\overline\nabla g)\;.\\ 
\end{aligned}
\end{equation}
Here, $\nabla=\begin{pmatrix}
    \nabla_x\\
    \nabla_v
\end{pmatrix}$ denotes the gradient on $\R^{2d}$. Further, 
$\overline\nabla g:= g'+g_*'-g-g_*$ recalling the notation \eqref{eq:shorthand}.
$\Lambda(f)$ is shorthand for $\Lambda(ff_*,f'f_*')$ with $\Lambda(s,t):=(s-t)/(\log s-\log t)$ denoting the logarithmic mean and finally the "divergence" $\overline\nabla\cdot$ is the adjoint in $L^2$ of $\overline\nabla$. More precisely, for a function $U:\R^{4d}\times S^{d-1}\to \R$ we set 
\begin{align*}
\overline \nabla \cdot U (x,v) := \frac14\int
\big[&U(x,x_*,v,v_*,\omega)+U(x_*,x,v_*,v,\omega)\\
-&U(x,x_*,v',v_*',\omega)-U(x_*,x,v_*',v',\omega) \big]\dd x_*\dd v_*\dd\omega\;.
\end{align*}
This choice of building blocks formally satisfies the assumptions of GENERIC and the fuzzy Boltzmann equation \eqref{FBE} is indeed formally equivalent to \eqref{eq:generic}. Namely, we have
\[
\mathsf L(f)\mathsf d\mathsf E(f) = -v\cdot \nabla_x f\;, \quad \mathsf M(f)\mathsf d\mathsf S(f) = Q_{\sf fuz}(f)\;.  
\]
We defer the verification of these claims to Appendix~\ref{app:GS}.
We note that this is in accordance with Theorem~\ref{thm:1} stating that the energy $\mathsf E$  is conserved and the entropy $\mathsf S$ is non-decreasing along solutions of the fuzzy Boltzmann equation. We also note that the structure above is an extension to the spatially inhomogeneous setting of the gradient flow structure for the spatially homogeneous Boltzmann equation found in \cite{Erb23}. 

After these formal considerations, we now present our main result giving a variational characterisation of the fuzzy Boltzmann equation inspired by \eqref{eq:J-gen}, \eqref{eq:generic-gen}.

In order to define the dual dissipation potentials $\mathsf R(\partial_t f-\mathsf L\mathsf d\mathsf E)=\mathsf R(\partial_t f+v\cdot\nabla_x f)$, along a  curve $(f_t)$ of probability densities, we look for a functions $U_t:\R^{4d}\times S^{d-1}\to\R$ such that the following \emph{transport-collsion rate equation} holds
\begin{equation}\label{eq:TCRE-intro}
\partial_t f +v\cdot \nabla_x f+\overline\nabla \cdot U = 0\;.
\end{equation} 
Note that for a solution to the fuzzy Boltzmann equation we have $U=ff_*-f'f_*'$. In general we think of $U$ as prescribing the rate at which collisions happen between the particles. Fix a pair of even, lower semi-continuous convex conjugated functions $\Psi, \Psi^{*}:\R\to[0,\infty)$ with $\Psi(0)=\Psi^{*}(0)=0$ and a 1-homogeneous concave function $\theta:[0,\infty)\times[0,\infty)\to [0,\infty)$ such that the compatibility condition
\[(\Psi^{*})'(\log s-\log t)\theta(s,t) = s-t\qquad\forall s,t>0\]
holds (see Assumption \ref{ass:gen-grad} for additional assumptions on $\Psi^{*}$ and $\theta$). Set we set
\begin{equation}\label{eq:potentials-intro}
\begin{aligned}
\mathcal R(f,U)&:=\frac14\int \Psi\Big(\frac{U}{\theta(f)Bk}\Big)\theta(f)Bk\dd v\dd\vs\dd\omega\;,\\
\mathcal D_{\Psi^{*}}(f)&:=\mathcal R^{*}(f, -DH(f)):=\frac14\int \Psi^{*}(-\bar\nabla \log f)\theta(f)Bk\dd v\dd\vs\dd\omega\;,
\end{aligned}
\end{equation}

where we have set $\theta(f):=\theta(ff_*,f'f_*')$. We refer to Sections \ref{sec:TCRE} and \ref{sec:potential} for the
precise construction where we study \eqref{eq:TCRE-intro} and \eqref{eq:potentials-intro} in a natural measure valued setting.
Then we have (see Thm.~\ref{thm:main} below):

\begin{theorem}\label{thm:main-intro}
Let $(f_t)_{t\in [0,T]}$ be a curve of probability measures and  $(U_t)$ be solving \eqref{eq:TCRE-intro} and assume that $\cH(f_0)<\infty$ and $(f_t)$ has finite and integrable in time moments of order 2 in $x$ and $2+\max(0,\gamma)$ in $v$. Then we have
\begin{equation*}\label{eq:LT-intro}
\cL_T(f,U):=\cH(f_T)-\cH(f_0) +\int_0^TD_{\Psi^*}(f_t)+\cR(f_t,U_t)\dd t \ge 0\;.
\end{equation*}
Moreover, we have $\cL_T(f,U)=0$ if and only if $(f_t)$ is a weak solution of the fuzzy Boltzmann equation \eqref{FBE} with   
$\int_0^T\cD_{\Psi^*}(f_t)\dd t<\infty$.
In this case, we have the entropy identity
\begin{equation*}
    \cH(f_t)-\cH(f_0)=-\int_0^t \cD(f_s)\dd s\quad\forall t\in[0,T].
\end{equation*}
\end{theorem}

The range of possible dual dissipation potentials includes two choices of particular interest. Choosing $\Psi(\xi)=\Psi^{*}(\xi)=\frac12 |\xi|^{2}$ and $\theta=\Lambda$ the logarithmic mean, we obtain the quadratic structure \eqref{eq:J-quad} related to the operator $\mathsf M$ from \eqref{eq:LM}. Another choice is
\[\theta(s,t)=\sqrt{st}\;,\qquad \Psi^{*}(\xi)=4\big(\cosh(\xi/2)-1\big)\;.\]
This particular variational structure seems to have been explicitly identified for the first time by Grmela, see e.g.~\cite[Eq. (A7)]{Grm93}.
For jump processes a similar structure is connected with the large deviations on the path level for the empirical measure of a growing number of independent particles, see e.g.~\cite{PRST20,PRS24}. Here, the resulting structure in the spatially homogeneous case can be related to the (only partially known) large deviation behaviour of the Kac particle system (see \cite{Hey23,BBBC22}) that gives rise to the homogeneous Boltzmann equation.

Let us briefly comment on the proof of Theorem \ref{thm:main-intro}. For the construction of the dual dissipation potentials we adapt the framework developed in \cite{PRST20} in the context of jump processes. The key to the variational characterisation of the fuzzy Boltzmann equation is a chain rule, giving the derivative of the entropy along solutions to the transport collision rate equation \eqref{eq:TCRE-intro} assuming only finiteness of the dissipation potentials, see Theorem \ref{thm:chainrule}. This requires a careful regularisation procedure. Here we follow and generalise the ideas developed in \cite{Erb23} for the homogeneous Boltzmann equation.

\subsection*{Organisation}
The rest of the paper is organised as follows. In Section \ref{sec:prelim} we collect some preliminaries. Existence and uniqueness of solutions to the fuzzy Boltzmann equation are discussed in Section \ref{analytic}. In Section \ref{sec:thm:2} we introduce the transport collision rate equation and dual dissipation potentials and establish the variational characterisation of the fuzzy Boltzmann equation. The appendix contains the formal verification of its GENERIC structure.

\subsection*{Acknowledgements}
Funded by the Deutsche Forschungsgemeinschaft (DFG, German Research Foundation) – Project-ID 317210226 – SFB 1283.

The authors are grateful for discussions with Helge Dietert and Andr\'e Schlichting.

\section{Preliminaries}\label{sec:prelim}

Let $d\geq 3$. Throughout the paper we make the following assumptions on collision kernels $B$ and $k$.

\begin{assumption}\label{CK}
The collision kernel $B:\R^d\times S^{d-1}\to \R_+$ is measurable, continuous w.r.t. the first variable, and invariant under the transformation \eqref{eq:pre-post} in the sense that 
   \begin{equation*}
\label{sys:ck}
    B(v-v_*,\omega)= B(v_*-v,\omega)=B(v'_*-v',\omega)\quad \forall v,v_*\in \R^d,~\omega\in S^{d-1}\;.
\end{equation*}
There existst $\mu\in (-\infty,1]$ and a constant $C_B>0$ such that 
\begin{equation*}
    C_B^{-1}\langle v\rangle ^\mu \le B(v,\omega)\le C_B\langle v\rangle ^\mu\quad \forall v\in \R^d,~\omega \in S^{d-1}\;.
\end{equation*}
The spatial kernel $k:\R^d\to \R_+$ is of the form
\begin{equation*}
k(z) = C\exp(-\gamma \jb{z})\;,
\end{equation*}
for some constants $\gamma,C>0$.
\end{assumption}

The following property will be crucial in the proof of the chain rule in Theorem \ref{thm:chainrule}. For $\beta>0$ we denote by $M_\beta$ the Gaussian density with variance $\beta$ on $\R^d$ given by
\[M_\beta(z) = (2\pi\beta)^{-d/2}\exp(-|z|^2/2\beta)\;.\]

\begin{lemma}\label{lem:kernel-convolve}
Under assumption \ref{CK} there exist a constant $C$ such that for any $\beta\in(0,1)$:
\begin{align}\label{eq:kernel-convolve-B}
B*_vM_\beta&\leq C\cdot B\;,\quad B^{-1}*_vM_\beta \leq C\cdot B^{-1}\;,\\\label{eq:kernel-convolve-k}
k*_x M_\beta &\leq C\cdot k\;, \quad k^{-1}*_xM_\beta\leq C\cdot k^{-1}\;.
\end{align}
\end{lemma}

\begin{proof}
We use the fact that for any $p\in \R$ and $,y\in\R^d$
  \begin{equation*}
    \label{eq:petree}
    \frac{\jb{x}^p}{\jb{y}^p}\leq 2^{|p|/2}\jb{x-y}^{|p|}\;,
  \end{equation*}
  known as Peetre's inequality. It readily implies that 
    \[\big(\jb{\cdot}^p*M_\beta\big)(v) \leq C \jb{v}^p\;,\]
for a constant $C$ depending only on $|p|$ and $\int |v|^{|p|}M(v)\dd v$.
Together with Assumption \ref{CK} on $B$ we immediately infer the claim \eqref{eq:kernel-convolve-B}.

To obtain the claim for the spatial kernel $k$ we use that there is a constant $C>0$ such that for any $\alpha\in(0,1)$:
\[
\big(\exp(\pm\gamma\langle \cdot \rangle)* M_\alpha\big )(x)\le C \exp(\pm\gamma\langle x \rangle)\;.
 \]
Indeed, using that $|x|\le\langle x\rangle \le|x|+1$ we obtain
\begin{align*}
&\exp(\mp\gamma\langle x \rangle)\big(\exp(\pm\gamma\langle \cdot \rangle)* M_\alpha\big )(x)\\
&=\int_{\R^d}\exp(\mp\gamma\langle x\rangle\pm\gamma\langle \sqrt\alpha z+x \rangle-|z|^2)\dd z
\le \int_{\R^d}\exp\big( \gamma(1+\sqrt\alpha |z|)-|z|^2\big)\dd z\\
&\le \int_{\R^d}\exp\big( \gamma\jb{z}-|z|^2\big)\dd z=:C\;.
    \end{align*}
In light of Assumption \ref{CK} on $k$ we obtain \eqref{eq:kernel-convolve-k}.

  \end{proof}

For fixed $\omega\in S^{d-1}$ we will denote by 
\begin{equation}\label{eq:pre-post2}
T_\omega:\R^{2d}\to\R^{2d}\;,\quad  (v,\vs)\mapsto(\vp,\vsp)\;,
\end{equation}
the transformation between pre- and post-collisional velocities with $\vp,\vsp$ given by \eqref{eq:pre-post}. Note that $T_\omega$ is involutive and has unit Jacobian determinant. A crucial property of the Gaussian convolution in the velocity variable is that it commutes with the change of variables between pre- and post-collisional velocities. 

More precisely, let us write
\begin{align*}
  V=(v,\vs),\quad V'=(\vp,\vsp) = T_\omega V\;. 
\end{align*}
For a function $F:\R^d\times\R^d\to\R$ we set
\begin{align*}
  T_\omega F (V) := F(T_\omega V)\;.
\end{align*}
We will write for short $F*M_\beta=F*_VM_\beta=F*_{(v,\vs)}M_\beta$, the independent convolution in both variables with the Gaussian $M_\beta$.

Then we have the following commutation property. For all $F\in L^1(\R^{2d})$ and each $\omega\in
  S^{d-1}$ and $\beta > 0$:
  \begin{align}\label{eq:commutation-scaling-convolution}
   (T_\omega F)*M_\beta = T_\omega (F*M_\beta) \;.
  \end{align}
 If $F=f\fS$ we have for short $(\fP\fSP)*M_\beta=(f*M_\beta)'(f*M_\beta)'_*$.

We refer to \cite[Prop.~4] {TV99} or \cite[Lem.~2.3]{Erb23} for the short proof. Note that we trivially have $(T_\omega F)*_XM_\alpha= T_\omega(F*_XM_\alpha)$.

We denote by $\cP(\R^{2d})$ the space of Borel probability
measures on $\R^{2d}$ equipped with the topology of weak convergence in
duality with bounded continuous functions. 
For $\mu\in \cP(\R^{2d})$ and $p\geq1$ we set
\[\cE_{p,q}(\mu) := \int_{\R^{2d}}\jb{x}^p+\jb{v}^q\dd\mu(x,v)\;,\]
and denote by $\cP_{p,q}(\R^{2d})$ the set of $\mu\in\cP(\R^{2d})$ such that $\cE_{p,q}(\mu)<\infty$.

We denote by $\cH(\mu)$ the
\emph{Boltzmann--Shannon entropy} defined for $\mu\in\cP(\R^{2d})$ by
\begin{align*}
  \cH(\mu) = \int f\log f \dd \cL\;,
\end{align*}
provided $\mu=f\cL$ is absolutely continuous w.r.t.~Lebesgue measure
$\cL$ and $\max(f\log f,0)$ is integrable, otherwise we set
$\cH(\mu)=+\infty$. We will also write $\cH(f)$ if $\mu=f\cL$.

The relative entropy of $\mu\in \cP(\R^{2d})$ w.r.t.~another probability measure $\nu\in\cP(\R^{2d})$ defined by 
\begin{align*}
  \cH(\mu|\nu) = \int f \log f\dd \nu\;. 
\end{align*}
provided $\mu=f\nu$ and $+\infty$ otherwise. By Jensen's inequality we have $\cH(\cdot|\nu)\geq 0$. Let us denote by $\mathsf m$ the standard Gaussian measure on $\R^{2d}$, i.e.~$\dd\mathsf m (x,v)=M_1(x)M_1(v)\dd x\dd v$. For any $\mu\in \cP_{2,2}(\R^{2d})$ we have
\begin{align*}
  \cH(\mu) = \cH(\mu|\mathsf m) - \frac12\mathcal E_{2,2}(\mu) -\frac{d}{2}\log(2\pi)\;.
\end{align*}
 Hence, we see that $\cH>-\infty$ on $\cP_{2,2}(\R^{2d})$.  Finally, we note that $\cH$ is lower semicontinuous w.r.t.~weak convergence plus convergence of second moments. I.e.~if $\mu_n\to \mu$ weakly in $\cP(\R^{2d}$ and $\cE_{2,2}(\mu_n)\to\cE_{2,2}(\mu)$ as $n\to\infty$, we have that 
 \[\cH(\mu)\leq \liminf_n \cH(\mu_n)\;.\]
This follows from the fact that $\cH(\cdot|\cdot)$ is lower semicontinuous w.r.t.~weak convergence.

\section{Solvability of the fuzzy Boltzmann equation}\label{analytic}

In this section, we study the solvability of the fuzzy Boltzmann equations \eqref{FBE}. We use the following notion of weak solution. Here and in the sequel we will denote the collision operator of the fuzzy Boltzmann equation by $Q$ instead of $Q_{\sf fuz}$ as in the introduction.

\begin{definition}[Weak solution]
For any given initial value $f_0\in L^1(\Do)$, we say $f_t(x,v)\in C([0,T]; L^1(\Do))$  is a weak solution of \eqref{FBE}, if
\begin{equation*}
\label{weakformulation}
    \begin{aligned}
    &\int_{\Do}\varphi(0) f_0\dd x\dd v-\int_0^T\int_{\Do}(\d_t+v\cdot\nabla_x)\varphi f_t\dd x\dd v\dd t \\
    &=-\frac14\int_0^T\int_{\G} \bar\nabla\varphi \big(f_t'(f_t)_*'-f_t(f_t)_*\big) Bk \dd\sigma \dd t
    \end{aligned}
\end{equation*}
for all $\varphi\in C^\infty_c([0,T)\times\Do)$ and $\dd\sigma=\dd x\dd x_*\dd v\dd v_*\dd\omega$ denotes the Hausdorff measure on $\R^{4d}\times S^{d-1}$.
\end{definition}
As a consequence of the symmetric assumption in Assumption~\ref{CK}, we have 
\begin{equation}
\label{sys}
\int_{\Do}Q(f_t)\varphi\dd x\dd v=-\frac14 \int_{\G} \bar\nabla\varphi \big(f_t'(f_t)_*'-f_t(f_t)_*\big) Bk\dd\sigma.
\end{equation}

For any fixed $p,\,q\ge 1$, we denote 
\begin{align*}
\|f\|_{L^1_{p,0}}:=\|\langle x\rangle ^pf\|_{L^1(\Do)}\quad\text{and}\quad \|f\|_{L^1_{0,q}}:=\|\langle v\rangle ^qf\|_{L^1(\Do)}.
\end{align*}
The functional space $L^1_{p,q}(\Do)$ consists all the functions such that $\|f\|_{L^1_{p,0}}+\|f\|_{L^1_{0,q}}<\infty$.

We have the following existence, uniqueness, and entropy dissipation results for the solutions of the Cauchy problem for the fuzzy Boltzmann equation \eqref{FBE}.

\begin{theorem}
\label{thm:1}
Let the collision kernel $B$ satisfy Assumption~\ref{CK}.
\begin{itemize}
    \item {\bf In the case of $\mu\in(-\infty,0]$:} If $f_0\in L^1(\Do;\R_+)$, then there exists a { unique} global-in-time weak solution $f_t\in C^1([0,T];L^1(\Do;\R_+)$ for any $T>0$ of \eqref{FBE}. The mass $\|f_t\|_{L^1(\Do)}$ is conserved.

   Furthermore, if  $f_0\in L^1_{0,2}(\Do)$, then the following energy and momentum conservation laws hold
    \begin{equation}
    \label{cl}
\int_{\Do}(v,|v|^2)f_t\dd x\dd v=\int_{\Do}(v,|v|^2)f_0\dd x\dd v\quad\forall t\in[0,T];
\end{equation}
If  $f_0\in L^1_{2,2}(\Do)$, then we have
    \begin{equation}
    \label{clx}
    \| f_t\|_{L^1_{2,0}(\Do)}\lesssim_T \|f_0\|_{L^1_{2,2}(\Do)} \quad\forall t\in[0,T];
\end{equation}

If $f_0\in L^1_{2,2}(\Do)$ and $\cH(f_0):=\int_{\Do} f_0\log f_0<\infty$, then we have the following entropy inequality
\begin{equation}
\label{in:cH}
    \cH(f_t)-\cH(f_0)+\int_0^t \cD(f_s)\dd s\le0\quad \forall t\in[0,T],
\end{equation}
and the bounded entropy dissipation
\begin{equation}
\label{bdd:D:0}
   \int_0^T \cD(f_s)\dd s<\infty,
\end{equation}
where the entropy dissipation functional is defined in \eqref{eq:def-dissipation}.

 \item  {\bf In the case of $\mu\in(0,1]$:}  If $f_0\in L^1_{2,2+\mu}(\Do)$ and  $\cH(f_0)<\infty$, 
    then there exists a unique global-in-time weak solution $f\in C^1([0,T];L^1(\Do))\cap L^\infty([0,T];L^1_{2,2+\mu}(\Do))$ for any $T>0$ of \eqref{FBE} such that the following conservation laws holds
    \begin{equation*}
\int_{\Do}(1,v,|v|^2)f_t\dd x\dd v=\int_{\Do}(1,v,|v|^2)f_0\dd x\dd v,
\end{equation*}  
the estimate  \eqref{clx} holds, we have
\begin{align}  
&\|f_t\|_{L^\infty([0,T];L^1_{0,2+\mu}(\Do))}\lesssim\|f_0\|_{L^1_{0,2+\mu}(\Do)}\exp\big(C_T\|f_0\|_{L^1_{0,2}(\Do)}\big),
\label{unbd:lem:2+mu}
\end{align} 
and the entropy inequality \eqref{in:cH} holds. In particular, the entropy dissipation is bounded \eqref{bdd:D:0}.



\end{itemize}
\end{theorem}

As part of the variational characterisation of solutions to the fuzzy Boltzmann equation in Theorem \ref{thm:main} below, we will show that the weak solution obtained above for initial data $f_0\in L^1_{2,2+\max(\mu,0)}(\Do)$ with $\cH(f_0)<\infty$ in fact satisfies the following entropy identity:
\begin{equation*}
\label{cH}
    \cH(f_t)-\cH(f_0)=-\int_0^t \cD(f_s)\dd s\quad\forall t\in[0,T].
\end{equation*}

We observe that the fuzzy collision term has the following priori bound
\begin{equation*}
    \label{PE}
\|Q(f)\|_{L^1}\le\|k\|_{L^\infty}\|B\|_{L^\infty}\|f\|_{L^1}^2.
\end{equation*}
As a consequence, the fuzzy Boltzmann equation \eqref{FBE} is structurally closer to the {homogeneous} Boltzmann equation 
\eqref{HBE}. Additionally, the transportation term $v\cdot\nabla_x f_t$ in the fuzzy Boltzmann equation can be handled by considering \eqref{FBE} along characteristic lines  
\begin{equation}
\label{FBE:charac}
\left\{
\begin{aligned}
&\d_t[f_t]=[Q(f_t)], \\
&[f_t]|_{t=0}=f_0,
\end{aligned}
\right.
\end{equation}
where we define the notation  
\begin{equation*}
[f_t](x,v):=f_t(x(t),v),\quad  x(t)=x+vt.
\end{equation*}

Based on the above observation, we can derive the existence, uniqueness and entropy results of fuzzy Boltzmann equation \eqref{FBE} by paralleling the arguments for homogeneous Boltzmann equations \eqref{HBE} in \cite{Ark72}. For completeness, we will sketch the proof here.

The reason why we split the cases $\mu\in(-\infty,0]$ and $\mu\in(0,1]$ is following:
\begin{align*}
0\le B(v,w) \le C_B \quad&\text{uniformly bounded, if }\mu\in (-\infty,0],\\
0\le B(v,w) \le C_B\langle v\rangle^\mu\quad&\text{polynomially growing, if }\mu\in (0,1].
\end{align*}
We will first treat the simpler case that the collision kernel is bounded in Section~\ref{bd}. Based on the bounded collision kernel results, we show the unbounded collision kernel case in Section~\ref{unbdd}.

  In the following, we fix an arbitrary time $T>0$. Without loss of generality, we assume the space kernel $0\le k\le 1$.

\subsection{Existence \& uniqueness of \texorpdfstring{$\mu\in(-\infty,0]$}{TEXT} case}\label{bd}

\medskip

{\it Existence.} 

We will first construct a sequence of approximate solutions, and then pass to the limit.

We consider the following approximation system with $f^1\equiv0$ 
\begin{equation}
\label{app}
\left\{
\begin{aligned}
&\d_t f^{n+1}+v\cdot \nabla_x f^{n+1}+c_0f^{n+1}=\bar Q(f^{n}),\\
&f^{n+1}|_{t=0}=f_0,\quad n\ge 1,
\end{aligned}
\right.
\end{equation}
where 
$c_0:=c\|f_0\|_{L^1(\Do)}$, and $\bar Q$ is defined as
\begin{equation}
\label{bar:Q}
    \begin{aligned}
   \bar  Q(f)= Q(f)+cf\int_{\Do}f.
    \end{aligned}
\end{equation}
The constant $c>0$ will be chosen large enough to ensure the positivity and monotonicity for $\bar Q$. 

The approximation system \eqref{app} on the characteristic lines reads as
\begin{equation}
\label{app:charac}
\left\{
\begin{aligned}
&\d_t [f^{n+1}]+c_0[f^{n+1}]=[\bar Q(f^{n})],\\
&[f^{n+1}]|_{t=0}=f_0,
\end{aligned}
\right.
\end{equation}
and the corresponding solutions are given by  
\begin{equation}
    \label{Duhamel}
[f^{n+1}_t]=e^{-c_0t}f_0+\int_0^{t}e^{-c_0(t-s)} [\bar Q(f^{n}_s)]\dd s.
\end{equation}

Based on the above observations, we have the following  properties of $f_t^n$:
\begin{itemize}
    \item \textit{Positivity $f^{n+1}\ge0$.}
    We use induction and assume  $[f^{n}]\ge 0$, then by definition
    \begin{align*}
        [\bar Q(f^{n})]
        &\ge c[f^{n}]\|f^{n}\|_{L^\infty_tL^1}+[Q(f^{n})]\\
        &\ge c[f^{n}]\|f^{n}\|_{L^\infty_tL^1}-[\int_{\Do\times S^{d-1}}f^{n}f^{n}_*Bk]\\
        &\ge (c-C_B)[f^{n}]\|f^{n}\|_{L^\infty_tL^1}>0,
    \end{align*}
     where we choose $c=2C_B$. Substituting $[\bar Q(f^{n})]\ge0$ to \eqref{Duhamel} implies the positivity of $f^{n+1}$.

\item \textit{Monotocity $f^{n+1}\ge f^n$.}
According to \eqref{Duhamel},  we only need to prove the monotonicity of $[\bar Q(\cdot)]$ (equivalently $\bar Q(\cdot)$), i.e. 
\begin{equation*}
0\le h\le g\Rightarrow  \bar Q(h)\le  \bar Q(g)\quad\forall (x,v)\in\Do.
\end{equation*}
Indeed, 
    \begin{align*}
        &\Big(\bar Q(g)-\bar Q(h)\Big)(x,v)\\
        &=\int_{\Do\times S^{d-1}} \big(g'(g-h)_*'+(g-h)'h_*'\big)Bk\dd x_*\dd v_*\dd w\\
        &\quad+\int_{\Do}
        \big((g-h)_*g+(g-h)h_*\big)\big( c-Bk\big)\dd x_*\dd v_*\dd w,
    \end{align*}
    which is positive since  $c=2C_B$.

    \item \textit{Boundedness of $L^1$-norm $\|f^n_t\|_{L^1(\Do)}\le \|f_0\|_{L^1(\Do)}$.} 
 We observe
    \begin{equation*}
\|f\|_{L^1}=\|[f]\|_{L^1}.  
    \end{equation*}
   We use induction and assume  $\|f^{n}_t\|_{L^\infty_tL^1}\le \|f_0\|_{L^1}$, then we integrate \eqref{app:charac} over $[0,t]\times \Do$ to derive 
    \begin{equation*}
    \begin{aligned} 
    &\|f^{n+1}_t\|_{L^1}=\|f_0\|_{L^1}+\underbrace{\int_0^t\int_{\Do}[ Q(f^{n}_s)]}_{=0}\\
    &+c\int_0^t\big(\int_{\Do} f^{n}_s\big)^2-\|f_0\|_{L^1}\int_{\Do}f^{n+1}_s\\
        &\le  \|f_0\|_{L^1}+c\int_0^t\int_{\Do} \underbrace{f^{n}_s-f^{n+1}_s}_{\le 0}.
    \end{aligned}
    \end{equation*}
    
\end{itemize}

By directly applying Levi's lemma (see e.g. \cite{Ark72}), $f^{n}_t$ converges to some $f_t$ in $L^1(\Do)$ for all $t\in[0,T]$. We use the dominated convergence theorem to pass to the limit ($n\to\infty$) in the  weak formulation of \eqref{app:K} to derive 
\begin{equation*}
    \label{weak:fk:1}
    \begin{aligned}
   & \int_{\Do} f_0 \varphi(0)-\int_0^T\int_{\Do} f_t (\d_t+v\cdot\nabla_x)\varphi +c_0\int_0^T\int_{\Do} f_t \varphi\\ 
   &=c\|f_t\|_{L^1(\Do)}\int_0^T\int_{\Do} f_t \varphi+\int_0^T\int_{\Do} Q(f_t) \varphi\quad\forall \varphi\in C^\infty_c([0,T)\times\Do).
\end{aligned}
\end{equation*}
 Now to verify $f_t$ is indeed a weak solution of \eqref{FBE}, we only need to show $\|f_t\|_{L^1(\Do)}=\|f_0\|_{L^1(\Do)}$ for all $t\in[0,T]$. Indeed, we integrate \eqref{app} over $\Do$ and pass to the limit ($n\to\infty$)  
\begin{equation*}
\begin{aligned}
\frac{\dd}{\dd t}\| f_t\|_{L^1} +c\|f_0\|_{L^1}\|f_t\|_{L^1}=c\|f_t\|_{L^1}^2.
\end{aligned}    
\end{equation*}
Since $0\le \|f_t\|_{L^1}\le \|f_0\|_{L^1}$, the above ODE has the only solution  $\|f_t\|_{L^1}=\|f_0\|_{L^1}$ for all $t\in[0,T]$. 

The Lipschitz continuity in time can be shown via
\begin{equation*}
\label{lip:t}
\|f_t-f_s\|_{L^1}\le|\int_s^t \int_{\Do}Q(f_r)\dd x\dd v\dd r |\le C|t-s|\|f_0\|_{L^1}^2.   
\end{equation*}

\medskip

{\it  Uniqueness.} 

Let $f_t,g_t\in L^\infty([0,T];L^1(\Do))$ both be the solutions of \eqref{FBE:charac} with the same initial value $f_0\in L^1(\Do)$. We define the difference $\dot f_t=f_t-g_t$, which satisfies
\begin{equation}
\label{diff}
\left\{
\begin{aligned}
&\d_t [\dot f_t]=[\int_{\Do\times S^{d-1}}Bk(f_t'\dot {(f_t)'_*}+\dot f_t'(g_t)_*'-\dot{(f_t)}_*f_t-\dot{f_t} (g_t)_*)\dd v_*\dd x_*\dd w]\\
&\dot{f_t}|_{t=0}=0.
\end{aligned}
\right.
\end{equation}

We multiply \eqref{diff} by $\sgn(\dot f_t)$ and integrate over $\Do$ to derive
\begin{align*}
    \frac{\dd}{\dd t}\int_{\Do} |\dot{f_t}|\dd x\dd v&\le  C_B (\|f_t\|_{L^1}+\|g_t\|_{L^1})\|\dot{f_t}\|_{L^1}\le  2C_B\|f_0\|_{L^1}\|\dot{f_t}\|_{L^1}.
\end{align*}
Gronwall's inequality implies then $\dot{f_t}=0$.

\medskip
{\it  Conservation laws and propagation of regularities.}

We first observe that 
\begin{equation*}
   \bar\nabla v=\bar\nabla|v|^2=0.
\end{equation*}
We recall the symmetric identity \eqref{sys}
\begin{equation*}
 \int_{\Do}Q(f)\varphi\dd x\dd v=-\frac14 \int_{\G} \bar\nabla\varphi (f'f_*'-ff_*) Bk\dd\sigma.
\end{equation*}
We take $\varphi=v,\,|v|^2$, then at least formally the conservation laws \eqref{cl} hold. 

To be rigorous, it suffices to verify $f_t\in L^1_{0,2}(\Do)$. Indeed, we multiply the approximation system \eqref{app} by $|v\wedge k|^2$, $k\in\N_+$ and integrate over $[0,t]\times\Do$. By repeating the similar induction argument for showing $\|f^n_t\|_{L^1}\le \|f_0\|_{L^1}$, one has $\|f^n_t\|_{L^1_{0,2}}\le \|f_0\|_{L^1_{0,2}}$. The monotonicity of $f^n_t$ ensures $\|f_t\|_{L^1_{0,2}}\le \|f_0\|_{L^1_{0,2}}$, then  the conservation laws \eqref{cl} follow. 

Concerning the bound \eqref{clx}, we multiply the fuzzy Boltzmann equation \eqref{FBE} by $\langle x\wedge k\rangle^2$ and integrate over $\Do$ to derive 
\begin{align*}
    \frac{\dd}{\dd t}\|\langle x\wedge k\rangle^2 f_t\|_{L^1}\le C(\|\langle x\wedge k \rangle^2 f_t\|_{L^1}+\|f_t\|_{L^1_{0,2}}),
\end{align*}
where we use the fact $\bar\nabla\langle x\wedge k \rangle^2=0$ (this is generally not true for $\bar\nabla| v\wedge k|^2$). Gronwall's inequality implies 
\begin{equation*}
    \|\langle x\wedge k\rangle^2 f_t\|_{L^1}\le e^{Ct}\|f_0\|_{L^1_{2,0}}+\int_0^t e^{C(t-s)}\|f_s\|_{L^1_{0,2})} \,ds\quad\forall t\in[0,T].
\end{equation*}
Then \eqref{clx} can be derived by Fatou's lemma.

\medskip

{\it Entropy inequality.} 

To show the entropy inequality \eqref{in:cH}, we construct an approximation sequence $f_t^{ijk}$  and verify the dissipation of the corresponding entropy, then pass to the limit.

Let $f_t^{ijk}$ with $i,j,k\in\N_+$ be a solution of 
\begin{equation}
\label{app:ijk}
\left\{
\begin{aligned}
&\d_t [f_t^{ijk}]+c_j[f_t^{ijk}]=[\bar Q^i(f_t^{ijk})],\\
&f^{ijk}_t|_{t=0}=f_0^{ik}:=\min(f_0+k^{-1}(\langle x\rangle+\langle v\rangle )^{-a},i)
\end{aligned}
\right.
\end{equation}
where $a>2d+2$ and 
\begin{align*}
     c_j:=c\int_{\Do}f_0+j^{-1}(\langle x\rangle+\langle v\rangle )^{-a}\dd x\dd v,\quad\bar Q^i(f):=\min(\bar Q(f),i).
\end{align*}
In the above, the constant $c>0$ and $\bar Q(f)$ are defined as in \eqref{bar:Q}, where $c$ is large enough such that $\bar Q$ is positive and monotonic.

We note that the approximation solutions $f^{ijk}_t$  are given by  \begin{equation}
    \label{Duhamel:ijk}
[f^{ijk}_t]=e^{-c_jt}f_0^{ik}+\int_0^{t}e^{-c_j(t-s)} [\bar Q^i(f^{ijk}_s)]\dd s,
\end{equation}
then the positivity and boundedness of $\bar Q^i$ ensure that
\begin{equation}
\label{bounds:ijk}
 C_T(k)\exp(-\langle x\rangle^2-\langle v\rangle^2 )\le f^{ijk}\le C_T(i)\quad\forall \,i,\,j,\,k\in\N_+.
\end{equation}
We observe that $\bar Q^i$ and $f^{i,k}_0$ are monotonically increasing as $i\to\infty$, and $c_j$ and  $f^{i,k}_0$ are monotonically decreasing as $j,k\to\infty$. By \eqref{Duhamel:ijk}, we have  
\begin{equation}
\label{monoton:ijk}
    f^{i_1j_1k_1}\le f^{i_2j_2k_2}\quad \forall \,i_1\le i_2,\,j_1\le j_2,\, k_1\ge k_2.
\end{equation}

Notice that $f^{jj\infty}\in L^1_{2,2}(\Do)$ is also a solution of the fuzzy Boltzmann equation \eqref{FBE} with the initial value $f_0+j^{-1}(\langle x\rangle+\langle v\rangle)^{-a}$. The monotonicity \eqref{monoton:ijk} ensures 
\begin{equation*}
  f^{ijk}\le f^{jj\infty}\quad\forall\,j\le k,
\end{equation*}
which implies that $\|f^{ijk}\|_{ L^1_{2,2}}\le \|f^{jj\infty}\|_{ L^1_{2,2}}$ for all $j\le k$.  

Then the moments estimates and bounds \eqref{bounds:ijk} allow us to multiply \eqref{app:ijk} by $[\log f^{i,j,k}]$ and integrate over $\Do$ to derive the approximation entropy inequality
\begin{equation}
\label{cH:ijk}
\begin{aligned}
&\cH(f^{ijk}_t)-\cH(f^{ik}_0)\le \int_0^t\int_{\Do}\Big(\bar Q(f^{ijk}_s)-c_j f^{ijk}_s\Big)\log f^{ijk}_s\\
&+\int_0^t\int_{\Do}\bar Q(f^{ijk}_s)-c_j f^{ijk}_s+\int_0^t\int_{\bar Q(f^{ijk}_s)>i}\Big(i-\bar Q(f^{ijk}_s)\Big)\log^- f^{ijk}_s,
\end{aligned}
\end{equation}
By using   $0\le f^{ijk}\le f^{jj\infty}$ ($\forall\, j\le k$) and $\|f^{jj\infty}_t\|_{L^1}=\|f^{jj\infty}_0\|_{L^1}$, the right-hand side of \eqref{cH:ijk} has the following bounds
\begin{align*}
&\int_0^t\int_{\Do}\Big(\bar Q(f^{ijk}_s)-c_j f^{ijk}_s\Big)\log f^{ijk}_s\\
&\le -\int_0^t\cD(f^{ijk}_s)+c\int_0^t\int_{\Do}f^{ijk}_s\log^- f^{ijk}_s\Big(\int_{\Do}f^{ijk}_s-f^{jj\infty}_s\Big),\\
&\int_0^t\int_{\Do}\bar Q(f^{ijk}_s)-c_j f^{ijk}_s\le c\int_0^t\int_{\Do}f^{ijk}_s\Big(\int_{\Do}f^{ijk}_s-f^{jj\infty}_s\Big)\le0,
\end{align*}
and 
\begin{align*}
&\int_0^t\int_{\bar Q(f^{ijk}_s)>i}\Big(i-\bar Q(f^{ijk}_s)\Big)\log^- f^{ijk}_s\\
&\le C_T(k)\int_0^t\int_{\bar Q(f^{ij\infty}_s)>i}\Big(\bar Q(f^{jj\infty})-i\Big)(\langle x\rangle^2 +\langle v\rangle^2)\to 0
\end{align*}
as $i\to \infty$, where we use the lower bound \eqref{bounds:ijk} in the last inequality. 

Then we conclude
\begin{equation*}
\label{cH:ijk:1}
\begin{aligned}
&\cH(f^{ijk}_t)-\cH(f^{ik}_0)+\int_0^t\cD(f^{ijk}_s)\dd s\\
&\le C_T(k)o(1)+c\int_0^t\int_{\Do}f^{ijk}_s\log^- f^{ijk}_s\Big(\int_{\Do}f^{ijk}_s-f^{jj\infty}_s\Big). 
\end{aligned}
\end{equation*}

We first consider $i\to\infty$. Since $f^{ijk}$ is monotonically increasing and has the following bounds
\begin{equation*}
    C_T(m)\exp(-\langle x\rangle^2-\langle v\rangle^2)\le f^{j,k}\le f^{jj\infty},
\end{equation*}
we denote $f^{jk}_t=\lim_{i\to\infty}f^{ijk}_t$ in $L^1(\Do)$(up to a subsequence) for all $t\in[0,T]$.
By Fatou's lemma and dominated convergence Theorem, one has $f^{jk}\in L^1_{2,2}(\Do)$, $f^{jk}\log f^{jk}\in L^1(\Do)$ and
\begin{equation}
    \label{cH:jk}
    \begin{aligned}
 &\cH(f^{jk}_t)-\cH(f^{jk}_0)+\int_0^t\cD(f^{jk}_s)\dd s\le c\int_0^t\int_{\Do}f^{jk}_s\log^- f^{jk}_s\Big(\int_{\Do}f^{jk}_s-f^{jj\infty}_s\Big),    
\end{aligned}
\end{equation}
where we use the weak lower semicontinuity $\cD$ as in Lemma~\ref{lem:lsc-diss}. 

Similarly, we denote $f^j:=\lim_{k\to\infty}f^{jk}$ in $L^1(\Do)$, where $f^j$ is monotonically decreasing and bounded by $f^{11\infty}$. By using 
    $x\log x\ge -y+x\log y$ for all  $x,y>0$, one has 
\begin{equation*}
    0\ge f^j\log^-f^j\ge -\exp(-(\langle x\rangle^2+\langle v\rangle^2))-(\langle x\rangle^2+\langle v\rangle^2)f^j\in L^1(\Do).
\end{equation*}
Then we pass to the limit in \eqref{cH:jk} to derive
\begin{equation}
\label{ent:j}  
\begin{aligned}
 &\cH(f^{j}_t)-\cH(f^{j}_0)+\int_0^t\cD(f^{j}_s)\dd s\le c\int_0^t\int_{\Do}f^{j}_s\log^- f^{j}_s\Big(\int_{\Do}f^{j}_s-f^{jj\infty}_s\Big).    
\end{aligned}
\end{equation}

Similarly, we denote $\bar f:=\lim_{j\to\infty} f^j$ in $L^1(\Do)$, where $\bar f\log \bar f\in L^1(\Do)$ and $\bar f\in L^1_{2,2}(\Do)$. One can verify that $\bar f=f$ is a weak solution of the fuzzy Boltzmann equation \eqref{FBE} with initial value $f_0$, and hence, by mass conservation law, we have
\begin{align*}
    \lim_{j\to\infty}\int_{\Do}f^j-f^{jj\infty}=\lim_{j\to\infty}j^{-1}\int_{\Do} (\langle x\rangle+\langle v\rangle)^{-a}=0. 
\end{align*}
We pass to the limit by letting $j\to\infty$ in \eqref{ent:j} to obtain the entropy inequality \eqref{in:cH} 
\begin{align}\label{in:ent-in-bdd}
 \cH(f_t)-\cH(f_0)+\int_0^t\cD(f_s)\dd s\le 0\quad \forall t\in[0,T].  
\end{align}

Notice that the moments estimate $\|f_t\|_{L^1_{2,2}}\lesssim_T\|f_0\|_{L^1_{2,2}}$ ensures
\begin{equation*}
    \cH(f_t)\gtrsim_T -\|f_0\|_{L^1_{2,2}}+1\quad\forall t\in[0,T],
\end{equation*}
see e.g. \cite{JKO98}. Hence, we have the entropy dissipation bound \eqref{bdd:D:0}
\begin{equation}
\label{bdd:D:est}
    \int_0^T\cD(f_s)\dd s\le \cH(f_0)+C_T(\|f_0\|_{L^1_{2,2}}+1).
\end{equation}

\subsection{Existence \& uniqueness of \texorpdfstring{$\mu\in(0,1]$}{TEXT} case }\label{unbdd}

To show the existence of weak solutions, we will first construct a sequence of approximation solutions corresponding to approximation bounded collision kernels, then pass to the limit.

We define the approximation fuzzy Boltzmann equations 
\begin{equation}
\label{app:K}
\left\{
\begin{aligned}
&\d_t f+v\cdot\nabla_x f=Q^m(f),\\
&f|_{t=0}=f_0,\quad m\in\N_+,
\end{aligned}
\right.
\end{equation}
where
\begin{equation*}
    \begin{aligned}
    Q^m(f):= \int_{\Do\times S^{d-1}}(f'f'_*-ff_*)B^mk,
    \end{aligned}
\end{equation*}
and $B^m$ denotes the approximation bounded collision kernels 
\begin{equation*}
    B^m(v,w):=\min\{B(v,w),m\}\quad m\in \N_+.
\end{equation*}
By Theorem~\ref{thm:1}-(1), for any fixed $k\in\N_+$ and initial value $f_0\in L^1_{2,2+\mu}(\Do)$, there exists a unique weak solution of \eqref{app:K}, denoted by $f^m\in C^1_tL^1$, such that
\begin{equation*}
\|f^m_t\|_{L^1_{0,2}}=\|f_0\|_{L^1_{0,2}},\quad \|f^m_t\|_{L^1_{2,0}}\lesssim_T\|f_0\|_{L^1_{2,0}}\quad\text{and}\quad\cH(f^m_t)\le \cH(f_0)   
\end{equation*}
for all $t\in[0,T]$. Furthermore, there is a uniform bound on $\|f^m_t\|_{L^1_{0,2+\mu}}$. Indeed, we use the Povzner inequality (see e.g. \cite{Pov62})
 \begin{equation*}
|\bar\nabla\langle v\rangle^{2+\mu}|\le C\big(\langle v\rangle^2\langle v_*\rangle^\mu+\langle v\rangle^{\mu}\langle v_*\rangle^2\big)
 \end{equation*}
to derive
\begin{align*}
 &\frac{\dd}{\dd t}\int_{\Do}\langle v\rangle^{2+\mu} f^m=\frac14 \int_{\G}\bar\nabla \langle v\rangle^{2+\mu} ((f^m)'(f^m)'_*-f^mf^m_*)B^m k\\
&\lesssim \|f^m\|_{L^1_{0,2+\mu}}\|f^m\|_{L^1_{0,\mu}}+\|f^m\|_{L^1_{0,2}}\|f^m\|_{L^1_{0,2\mu}}\lesssim \|f^m\|_{L^1_{0,2+\mu}}\|f_0\|_{L^1_{0,2}}.
\end{align*}
Gronwall's inequality implies then \eqref{unbd:lem:2+mu} w.r.t $f^m$.

By Dunford-Pettis Theorem (see e.g. \cite{Edw65}), the  sequence $\{f^m\}$
is weakly compact in $L^1(\Do)$ for any $t\in[0,T]$. The Cantor diagonal argument ensures that there exists a subsequence (still denoted by $f^m$) such that $f^m_t\rightharpoonup f_t$ for any $t\in \Q\cap[0,T]$. The convergence can be extended to $[0,T]$ thanks to the Lipschitz continuity of $f^m_t$ in time. Moreover, we have 
\begin{equation*}
\lim_{m\to\infty} \int_{\Do}f^m_t\varphi=\int_{\Do} f_t\varphi
\end{equation*}
for all $\varphi$ such that $(\langle v\rangle^{-k_1}+\langle x\rangle^{-k_2})\varphi\in L^\infty(\Do)$ with $0\le k_1<2+\mu$ and $0\le k_2<2$. Then conserved mass, momentum and kinetic energy follow. Similar to the bounded collision kernel case, one can verify that the limit $f_t$ is a weak solution for \eqref{FBE}. The monotonicity of $\{f^m_t\}$ implies estimates \eqref{clx} and \eqref{unbd:lem:2+mu}. 
The uniqueness can be derived similarly to the bounded collision kernel cases by considering the difference equation \eqref{diff}.




By using the weak lower semicontinuity of $\cH$ and $\cD$, see Lemma~\ref{lem:lsc-diss}, we can pass to the limit in the entropy inequality \eqref{in:ent-in-bdd} for the approximate solutions $f^m$ and obtain the entropy inequality for the weak solution $f$. As in the case $\mu\in (-\infty,0)$ we conclude the entropy dissipation bound \eqref{bdd:D:0}.


\begin{remark}
In this remark, we summarize some existence, uniqueness, and moments estimates results for homogeneous Boltzmann equation \eqref{HBE} in weaker settings compared to Theorem~\ref{thm:1}. We expect similar results in the fuzzy Boltzmann case.

Concerning the homogeneous Boltzmann equation with the collision kernels of the following form 
\begin{equation*}
     B(v,w)=b(\theta)|v|^\mu,\quad 0\le\mu\le 1,
      \end{equation*} 
       where $\theta$ denotes the deviation angle and $b(\theta)\sin\theta\in L^1([0,\frac{\pi}{2}])$:
\textcite{MW99}  showed that if the initial value is only in $ L^1_{0,2}(\R^3)$, then  there exists a unique solution $f_t\in L^\infty_tL^1_{0,2}$. Furthermore, if $f_0\in L^1_{0,p}$, for any  $p>2$, then the corresponding solutions $f_t\in L^\infty_tL^1_{0,p}\cap L^1_{\loc,t}L^1_{0,p+\mu}$; Furthermore, \textcite{Wen99} has shown that the assumption $f_0\in L^1_{0,2}$ is sharp for uniqueness. 

There are many works devoted to the moment estimates for homogeneous Boltzmann equations, see e.g. \cite{Des93,Gus86,Gus88,Wen93,Wen94,Wen97}. We emphasize that \textcite{Lu99} gave a necessary and sufficient condition for the initial value, ensuring $f_t\in L^1_{\text{loc},t}L^{1}_{0,2+\mu}$.
\end{remark}

\section{Variational characterisation of the fuzzy Boltzmann equation}
\label{sec:thm:2}

In this section, we will prove the variational characterisation of the fuzzy Boltzmann equation given in Theorem~\ref{thm:main}. We will first discuss the transport collision rate equation and then establish a chain rule for the entropy along curves satisfying the transport collision rate equation in Theorem~\ref{thm:chainrule}. Finally, we use the chain rule to prove Theorem~\ref{thm:main}. 

\subsection{Transport collision rate equation}\label{sec:TCRE}

Recall that $\cP(\R^{2d})$ denotes the set of Borel probability measures on $\R^{2d}$ equipped with the topology of weak convergence of convergence in duality with bounded continuous functions. We set $\Omega=\R^{4d}\times S^{d-1}$ and denote by $\cM(\Omega)$ the space of signed Borel measures with finite total variation on $\Omega$ equipped with the weak* topology of convergence in duality with continuous functions vanishing at infinity.  We let  $\sigma$ denote the Hausdorff measure on $\Omega$, i.e.
\begin{equation*}\label{eq:sigma}
\dd \sigma(x,\xs,v,\vs,\omega)= \dd x\dd \xs\dd v\dd \vs\dd\omega\;.
\end{equation*}

In the following we consider curves $(\mu_t)_{t\in[0,T]}$ in $\cP(\R^{2d})$ driven by a transport and collision mechanism according to a rate of collisions $(\cU_t)_{t\in[0,T]}$ via the transport collision rate equation
\begin{equation}\label{eq:TCRE}
\partial_t \mu_t + v\cdot \nabla_x\mu_t+\overline\nabla\cdot \cU_t=0\;.
\end{equation}

\begin{definition}\label{def:TCRE}
A pair $(\mu,\cU)$ is called a \emph{solution to the transport collision rate equation} if the following hold:
\begin{itemize}
\item[(i)] $\mu:[0,T]\to \cP(\R^{2d})$ is weakly continuous;
\item[(ii)] $(\cU_t)_{t\in[0,T]}$ is a Borel family of measures in $\cM(\Omega)$;
\item[(iii)] $\int_0^T\dd|\cU_t|(\Omega) \dd t <\infty$; 
\item[(iv)] for any $\varphi\in C^\infty_c(\R^{2d})$ we have in the sense of distributions:
\begin{align*}\label{eq:TCRE-dist}\frac{\dd}{\dd t}\int_{\R^{2d}} \varphi \dd\mu_t - \int_{\R^{2d}}v\cdot \nabla_x\varphi\dd\mu_t = \frac14 \int _\Omega \overline\nabla \varphi \dd\cU_t\;.
\end{align*}
\end{itemize} 
We denote the set of such pairs $(\mu,\cU)$ by $\TCRE_T$. 
\end{definition}

If $\mu$ and $\cU$ have densities $f$ and $U$ w.r.t. Hausdorff measure we will also write $(f,U)\in \TCRE_T$. 
Given $\cU\in \cM(\Omega)$, we will denote by $\overline{\nabla}\cdot \cU$ the signed measure given by 
\[\int_\Omega \varphi \dd\big(\overline{\nabla}\cdot \cU) =\frac14\int_\Omega\overline{\nabla}\varphi \dd\cU\;.\]

\begin{remark}\label{rem:generaltest}
  Under additional integrability assumptions on $\cU$, the transport collision rate equation can also be tested against more general test functions. More precisely, assume that $(\mu,\cU)\in\TCRE_T$. Then \eqref{eq:TCRE} holds for all $\varphi:\R^{2d}\to\R$ that are differentiable in $x$ and satify
 \begin{align}\label{eq:finite-test}
 \int_0^T |v|\big(|\varphi|+|\nabla_x\varphi|\big)\dd\mu_t\dd t &<\infty\;,
\quad  \int_0^T \big(|\varphi|+|\varphi_*|+|\varphi'|+|\varphi_*'|\big)\dd|\cU_t|\dd t<\infty\;,
 \end{align}
 as can be checked straightforwardly by approximating with $\varphi$ with functions in $C^\infty_c(\R^{2d})$.
Moreover, we have for such $\varphi$ and any $0\leq t_0\leq t_1\leq T$:
\[
\int \varphi \dd\mu_{t_1}-\int \varphi\dd\mu_{t_0} = \int_{t_0}^{t_1}\int v\cdot\nabla_x\varphi \dd\mu_r\dd r +  \int_{t_0}^{t_1}\int \overline\nabla\varphi \dd\cU_r\dd r\;.
\]
In particular, the map $t\mapsto\int \varphi\dd\mu_t$ is continuous.
\end{remark}
  
\subsection{Dual dissipation potentials}\label{sec:potential}
In this section we will introduce pairs of dual dissipation potentials that are the basic building blocks for the variational characterisation of the fuzzy Boltzmann equation.

We start by recalling basic results on integral functionals on measures that will be used in the sequel.

Let $\mathcal X$ be locally compact Polish space. We denote by $\cM(\mathcal X;\R^n)$
the space of vector-valued Borel measures with finite variation on
$\mathcal X$. It will be endowed with the weak* topology of convergence in
duality with $C_0(\mathcal X;\R^n)$, i.e.~continuous functions vanishing at
infinity. Let $f:\R^n\to [0,\infty]$ be a convex, lower
semicontinuous, and superlinear and let $\sigma$ be a non-negative finite Borel measure on $\mathcal X$. Define on
$\cM(\mathcal X;\R^n)$ the functional $\mathcal F_\alpha(\cdot|\sigma)$ via
  \begin{align}\label{eq:int-funct}
    \cF_{\alpha}(\gamma | \sigma)=\int_{\mathcal X}\alpha\left(\frac{\dd\gamma}{\dd\sigma}\right)\dd\sigma\;.
  \end{align}
  and set $\mathcal F_{\alpha}(\gamma|\sigma)=+\infty$ if $\lambda$ is not absolutely continuous w.r.t.~$\gamma$. Note that the definition is independent of the choice of $\sigma$ if $\alpha$ is positively $1$-homogeneous, i.e.~we have $\alpha(\lambda r)=\lambda \alpha(r)$ for all $r\in\R^{n}$ and $\lambda\geq 0$. We will write $\mathcal F_{\alpha}(\cdot)$ instead of $\mathcal   F_{\alpha}(\cdot|\sigma)$ in this case.
By \cite[Thm.~3.4.3]{Bu89} the functional $\mathcal F_\alpha(\cdot|\sigma)$ is sequentially lower semicontinuous w.r.t. the weak* topology.

In the following we fix a pair of functions satisfying

\begin{assumption}\label{ass:gen-grad}
Let $\Psi^*$,  called \emph{dual dissipation density}, $\theta$, \emph{flux density map}, be such that
\begin{enumerate}
\item the function $\Psi^*:\R\to[0,\infty)$ is convex, differentiable, superlinear and even with $\Psi^*(0)=0$ and satisfies $\Psi^*(r)\leq c_1\exp(c_2\cdot r)$ for constants $c_1,c_2>0$; 
\item the function $\theta:[0,\infty)\times[0,\infty)\to[0,\infty)$ is continuous, concave, and satisfies
\begin{itemize}
\item symmetry: $\theta(r,s)=\theta(s,r)$ for all $s,r\in [0,\infty)$;
\item positivity, normalisation: $\theta(s,t)>0$ for $s,t>0$ and $\theta(1,1)=1$;
\item positive $1$-homogeneity: $\theta(\lambda r,\lambda s) = \lambda\theta(r,s)$ for all $r,s\in[0,\infty)$ and $\lambda\geq 0$;
\item monotonicity: $\theta(r,t)\leq\theta(s,t)$ for all $0\leq r\leq s$ and $t\geq0$;
\item behavior at $0$: $\theta(0,t)=0$ for all $t\in[0,\infty)$;
\end{itemize}
\item in addition there holds
\begin{itemize}
\item compatibility: $(\Psi^{*})'(\log s-\log t) \theta(s,t) = s-t$ for all $s,t>0$;
\item there exists a convex lower semi-continuous function $G_{\Psi^{*}}:[0,\infty)\times[0,\infty)\to[0,\infty)$ such that 
\[\frac14\Psi^{*}(\log t-\log s)\theta(s,t) =
  G_{\Psi^{*}}(s,t)\quad\forall s,t>0\;,\]
and such that $G_{\Psi^*}(s,t)=0$ if and only if $s=t$. 
\end{itemize}
\end{enumerate}
\end{assumption}

Let $\Psi:\R\to\R$ be the convex conjugate of $\Psi^{*}$ and note that it is strictly convex, strictly increasing, superlinear and even with $\Psi(0)=0$.
As a direct consequence of convexity, we note that 
\begin{equation}\label{eq:psi-inc}
r\mapsto \Psi(r)\frac{1}{r} \quad \text{ is non-decreasing}\;.
\end{equation}

 \begin{remark}\label{rem:psi-psi*-id}
 Note that by the convex duality of $\Psi$ and $\Psi^{*}$, for $s,t> 0$ we have the estimate
\begin{align}\nonumber
\frac14\left|(\log t-\log s)w\right| &= \frac14\left|(\log s-\log t)\frac{w}{\theta(s,t)}\right|\theta(s,t)\\\nonumber
 &\leq \frac14\Psi\big(\frac{w}{\theta(s,t)}\big)\theta(s,t) +\frac14 \Psi^{*}\big(\log s-\log t\big)\theta(s,t)\\\label{eq:dual-est}
&= \frac14\Psi\big(\frac{w}{\theta(s,t)}\big)\theta(s,t) + G_{\Psi^{*}}(s,t)\;.
\end{align}
Moreover, we have equality
\begin{align*}
\frac14(\log t-\log s)w &= -\frac14\Psi\big(\frac{w}{\theta(s,t)}\big)\theta(s,t) - G_{\Psi^{*}}(s,t)
\end{align*}
 if and only if $w=(\Psi^{*})'\big(\log s-\log t\big)\theta(s,t)$ and hence by Assumption \ref{ass:gen-grad} if and only if $w=s-t$.
 
 Further we note that convexity of $\Psi^*$ implies $\Psi^*(r)\leq r (\Psi^*)'(r)$ and hence we get the bound
 \begin{align}\nonumber
  \Psi^{*}\big(\log s-\log t\big)\theta(s,t)
 & \leq 
 (\log s-\log t\big)(\Psi^{*})'\big(\log s-\log t\big)\theta(s,t)\\\label{eq:bd-psi-star}
 &= (\log s-\log t\big)(s-t)\;,
 \end{align}
 where we used Assumption \ref{ass:gen-grad} (3) in the last step.
 \end{remark}

\begin{remark}\label{rem:mean}
The assumptions on $\theta$ imply that $\theta$ is bounded above by the arithmetic mean:
\begin{equation}\label{eq:arithmetic}
\theta(s,t)\leq \frac{s+t}{2}\qquad \forall s,t\geq 0\;.
\end{equation}
\end{remark}

Given $\mu\in \cP(\R^{2d})$, we define the non-negative measures $\mu^1,\mu^2$ in $\cM_+(\Omega)$ given by
 \begin{equation}\label{eq:defmu12}
  \begin{split}
  \dd\mu^1(x,\xs, v, \vs, \omega) &:= B(v-\vs,\omega)k(x-\xs)\dd\mu(x,v)\dd\mu(\xs,\vs)\dd\omega\;,\\
  \mu^2 &:= T_\#\mu^1\;,
  \end{split}
\end{equation}
where $T$ is the change of variables
$(x,\xs,v,\vs,\omega)\mapsto (x,\xs,T_\omega(v,\vs),\omega)$ between pre- and post-collisional variables defined in \eqref{eq:pre-post2}. Furthermore, we define the measure
 $\nu_{\mu}\in\cM_{+}(\Omega)$ via 
\begin{equation}\label{eq:edge-measure}
\nu_{\mu}:=\theta\big(\frac{\dd\mu^{1}}{\dd\eta},\frac{\dd\mu^{2}}{\dd\eta}\big)\eta\;,
\end{equation}
where $\mu^{1},\mu^{2}$ are given by \eqref{eq:defmu12} and $\eta$ is any measure such that $\mu^{1},\mu^{2}\ll \eta$. Due to the $1$-homogeneity of $\theta$ the definition is independent of $\eta$. Note that if $\mu$ is absolutely continuous w.r.t.~Lebesgue measure with density $f$, then we have 
\begin{equation*}
\nu_{\mu} = \theta(f \fS,\fP \fSP) B k \sigma
\end{equation*}
 with $\sigma$ the Hausdorff measure on $\Omega$.

We can now define the primal and dual dissipation potentials.

\begin{definition}\label{def:diss-pot} Given measures $\mu\in \cP(\R^{2d})$, $\cU\in \cM(\Omega)$ we define
\begin{align*}
\mathcal R(\mu,\cU):=\frac14 \int_{\Omega}\Psi\Big(\frac{\dd \cU}{\dd \nu_{\mu}}\Big)\dd\nu_{\mu}\;,
\end{align*}
provided $\cU\ll \nu_{\mu}$ and $\mathcal R(\mu,\cU)=+\infty$ else. 
Given moreover a measurable function $\xi:\Omega\to\R$ we define
\begin{align*}
 \mathcal R^{*}(\mu,\xi):=\frac14\int_{\Omega}\Psi^{*}\big(\xi\big)\dd\nu_{\mu}\;.
\end{align*}
Finally, we define
\begin{equation*}
D_{\Psi^{*}}(\mu):= \int_{\Omega} G_{\Psi^{*}}(f\fS,\fP\fSP)B k \dd \sigma\;,
\end{equation*}
provided $\mu$ is absolutely continuous with density $f$ and set $D_{\Psi^{*}}(\mu)=+\infty$ otherwise.
\end{definition} 

Note that the functional $D_{\Psi^{*}}$ is formally given by
\begin{equation*}
D_{\Psi^{*}}(\mu)= \mathcal R^{*}(\mu, -\overline\nabla \log f)\;,
\end{equation*}
provided $\mu$ has density $f$. 

\begin{remark}
Let $\mu\in \cP(\R^{2d})$ and $\cU\in \mathcal M(\Omega)$ be such that $\cR(\mu,\cU)<\infty$. By definition $\cU$ is absolutely continuous w.r.t. $\nu_\mu$. Hence, if $\mu$ is absolutely continuous w.r.t. Lebesgue measure with density $f$, then $\cU_{t}$ is absolutely continuous wr.t. to the Hausdorff measure $\sigma$ on $\Omega$ with a density $U=W\theta(f)Bk$, where $\theta(f):=\theta(f\fS,\fP\fSP)$ for a suitable function $W$ and we have
\[\cR(\mu,\cU) = \int \Psi(W)\theta(f) Bk \dd\sigma\;.\]
In particular,  the set of $x,\xs,v,\vs,\omega$ where $\theta(f)=0$ is negligible for $\cU$.
\end{remark}

\begin{lemma}\label{lem:lsc-diss}
The functionals $\cR$ and $D_{\Psi*}$ are convex and lower semicontinuous. More precisely, for any sequences $(\mu_n)$ in $\cP(\R^{2d})$ converging weakly to $\mu$ and $(\cU_n)$ in $\cM(\Omega)$ converging weakly* to $\cU$ we have
\[\cR(\mu,\cU)\leq \liminf_n \cR(\mu_n,\cU_n)\;,\qquad D_{\Psi*}(\mu)\leq \liminf_n D_{\Psi*}(\mu_n)\;.\]
\end{lemma}

\begin{proof}
We note that the functionals $\cR(\cdot,\cdot)$ and $D_{\Psi*}(\cdot)$ can be rewritten as 
integral functionals of the form \eqref{eq:int-funct}. Namely, we have $\mathcal R(\mu,\cU)=\mathcal F_{G_\Psi}\big((\mu^1,\mu^2,\cU)\big)$, where the $1$-homogeneous, convex and lower semicontinuous function $G_\Psi:[0,\infty)\times[0,\infty)\times\R\to[0,\infty]$ is defined by
\begin{equation}\label{eq:alpha}
G_\Psi(s,t,u):=\begin{cases}
\frac14\Psi\big(\frac{u}{\theta(s,t)}\big)\theta(s,t)\;, & \theta(s,t)\neq 0\;,\\
0\;, & \theta(s,t) = 0 \text{ and } u=0\;,\\
+\infty\;, &\theta(s,t)=0 \text{ and } u\neq 0\;.
 \end{cases}
\end{equation} 
Similarly, we have $D_{\Psi*}(\mu)=\cF_{G_{\Psi*}}
\big((\mu^1,\mu^2)\big)$. This yields the convexity of the two functionals. Moreover, recalling \cite[Thm.~3.4.3]{Bu89}, we obtain that they are lower semicontinous w.r.t. weak* convergence of $\mu_n^1$, $\mu_n^2$, and $\cU_n$. To conclude, it suffices to note that under our assumptions on the kernels $B$, $k$, the weak convergence of $\mu_n$ to $\mu$ implies the weak* convergence of $\mu^1_n$ and $\mu^2_n$ to $\mu^1$ and $\mu^2$.
\end{proof}

\begin{example}\label{ex:quad-gen}
We highlight two examples of dual dissipation potentials compatible with our Assumption \ref{ass:gen-grad}.

\begin{itemize}
 \item \emph{Quadratic gradient structure:} Choose 
 \[ \Psi^{*}(r)=\Psi(r)=\frac12 r^{2}\;.\]
as well as
\[
\theta(s,t) 
= \int_0^1s^\alpha t^{1-\alpha}\dd\alpha
\]
the logarithmic mean and note that for $s,t>0$ we have
\[
\theta(s,t)=\frac{s-t}{\log s- \log t}\;.
\]
In this case we have
\[
G_{\Psi*}(s,t)
=
\begin{cases}
\frac12 (s-t)(\log s-\log t)\;, & s,t>0\;,\\
0\;, & s=t=0\;,\\
+\infty\;, & \text{else}\;.
\end{cases}
\]
This yields $D_{\Psi^{*}}(\mu)=\frac12D(\mu)$.
 
 \item \emph{$\cosh$ structure:} Let us set
 \[\theta(s,t) = \sqrt{st}\;,\qquad \Psi^{*}(\xi)=4\big(\cosh(\xi/2)-1\big)\;.\]
 Then we obtain 
 \[\Psi(s)=2s\log\Big(\frac{s+\sqrt{s^{2}+4}}{2}\Big) -\sqrt{s^{2}+4}+4\;,\]
 as well as
 \[G_{\Psi^{*}}(s,t) =\frac12\big(\sqrt{s}-\sqrt{t}\big)^{2}.\]
 \end{itemize}
 \end{example}

Finally, we note that an integrability estimate that will be needed later on.

\begin{lemma}\label{lem:integrability-U-R}
  Let $(\mu,\cU)\in\TCRE_T$ be such that
   $\cR(\mu,\cU)$ is finite and let $g:\Omega\to[0,\infty]$ be measurable. Then, we have
  \begin{align*}
   \int h \dd|\cU|
   \leq  \cR(\mu,\cU) + \frac12\int \big[\Psi^*(h)+ T_\omega\Psi^*(h)\big] Bk\dd \omega\dd\mu(x,v)\dd\mu(\xs,\vs)\;.
  \end{align*}
\end{lemma}

\begin{proof}
Since $\cR(\mu,\cU)$ is finite, we can write $\cU=U\nu_\mu$. Now the claim follows directly from the duality estimate $h |U|\leq \Psi(U)+\Psi^*(h)$, the definition of $\cR$ and the bound \eqref{eq:arithmetic}.
\end{proof}

\begin{corollary}\label{cor:special h}
Set $\xi(x,v)=1+\log(\jb{x}+\jb{v})$. We claim that for any $p>0$ and $q>\mu_+$ there exists a constant $C$ such that for any $(\mu,\cU)$ with $\cR(\mu,\cU)$ finite, we have:  
\begin{equation*}\label{eq:bd-xi-U}
\int \big(\xi+\xi_*+\xi'+\xi_*' \big)\dd|\cU| \leq C\cR(\mu,\cU)+ C\cE_{p,q}(\mu)\;.
\end{equation*}
\end{corollary}
\begin{proof}
From Lemma \ref{lem:integrability-U-R} we get for any $\varepsilon>0$:
\begin{align*}
\int \big(\xi+\xi_*+\xi'+\xi_*' \big)\dd|\cU| &\leq \frac{1}{\varepsilon} \cR(\mu,\cU) + \frac{1}{\varepsilon} \int \Psi^*\Big(\varepsilon\big(\xi+\xi_*+\xi'+\xi_*' \big)\Big)Bk \dd\omega\dd\mu\dd\mu_*\;.
\end{align*}
From the exponential growth assumption on $\Psi^*$ and the bound \eqref{CK} and $B$ we deduce for $\varepsilon$ sufficiently small:
\begin{align*}
&\int \Psi^*\Big(\varepsilon\big(\xi+\xi_*+\xi'+\xi_*' \big)\Big)Bk \dd\omega\dd\mu\dd\mu_*\\
&\leq 
C\int \big(\jb{x}+\jb{v}\big)^s\big(\jb{\xs}+\jb{\vs}\big)^s\big(\jb{x}+\jb{\vp}\big)^s\big(\jb{\xs}+\jb{\vsp}\big)^s\jb{v-\vs}^{\mu_+}\dd\omega \dd\mu\dd\mu_*\\
&\leq C \cE_{p,q}(\mu)\;.
\end{align*}
Here, $s$ can be made arbitrarily small by choosing $\varepsilon$ sufficiently small and the last bound is readily obtained using H\"older and Young inequality (the constant $C$ changes from line to  line).
\end{proof}

\subsection{Chain rule for the entropy}

The following result is the key ingredient to the variational characterisation of the fuzzy Boltzmann equation.

\begin{theorem}[Chain rule]
\label{thm:chainrule}
Let $(\mu,\cU)\in \TCRE_T$ with $(\mu_t)_t\subset\cP_{2,2+\mu_+}(\R^{2d})$ such that $\cH(\mu_0)<\infty$ and
\[\int_0^T\int_{\R^{2d}}|x|^2+|v|^{2+\mu_+}\dd \mu_t\dd t<\infty\;,\]
where $\mu_+=\max(0,\mu)$. Further assume that
\begin{equation}\label{eq:chain-ass}
\int_0^TD_{\Psi^*}(\mu_t)\dd t<\infty\;,\qquad \int_0^T \mathcal R(\mu_t,\cU_t)\dd t<\infty\;.
\end{equation}
Then $t\mapsto\cH(\mu_t)$ is absolutely continuous and we have
\begin{equation}\label{eq:chainrule}
      \frac{\dd}{\dd t} \cH(\mu_t)=\frac14\int_{\{\theta(f_t)>0\}}\overline\nabla \log( f_{t})\dd\cU_t\;,\quad \text{ for a.e.}\quad t\in [0,T]\;, 
    \end{equation}
    where $f_t$ denotes the density of $\mu_t$ w.r.t.~Lebesgue measure.
\end{theorem}

Note that the integral in \ref{eq:chainrule} is well-defined since $\theta(f)>0$ implies that $f$, $\fS$, $\fP$, and $\fSP$ are positive. Also we recall that the set $\{\theta(f_t)=0\}$ is negligible for $\cU_t$ for a.e. $t$.

A similar chain rule has been established in \cite{Erb23} in the spatially homogeneous situation for curves solving a collision rate equation. The main difference here lies in the transport term $v\cdot\nabla_x f_t$. We will follow a similar strategy employing a careful regularisation of the curve. Additional care is needed to control the transport term.

The proof of Theorem \ref{thm:chainrule} will be carried out in the rest of this subsection. We proceed in 4 steps: First we regularise the pair $(\mu,\cU)$ and provide estimates on the regularised quantities. Then we verify the chain rule for the regularised pair. Finally, we pass to the limit as the regularisation vanishes.

\medskip
\textbf{Step 1: Regularisation.}

We regularise the curve $(\mu,\cU)$ in several steps.

(1) \underline{Regularization in  $x$.}
\smallskip

Let $M_\alpha(v)=(2\pi\alpha)^{-d/2}\exp(-|v|^2/2\alpha)$ be the Gaussian density in $\R^d$ with variance $\alpha>0$.
We regularize $\mu_t$ and $\cU_t$ by convolution with $M_\alpha$ in the $x$ and $(x,x_*)$ variables, respectively, and put
 \begin{align*}
\mu_t^\alpha:=\mu_t *_x M_\alpha  \quad\text{and}\quad
\cU^{\alpha}_t=\cU_t *_{(x,x_*)}M_\alpha\;.
 \end{align*}
 More precisely, for test functions $\varphi:\R^d\times\R^d\to\R$ and $\Phi:\Omega\to \R$ we have
 \begin{align*}
 \int \varphi\dd\mu^\alpha_t &= \int \varphi(x-y,v)M_\alpha(y)\dd\mu(x,v)\dd v\;,\\
 \int \Phi\dd\cU^\alpha_t &= \int \Phi(x-y,x_*-y_*,v,v_*,\omega)M_\alpha(y)M_\alpha(y_*)\dd\cU_t(x,x_*,v,v_*,\omega)\dd y\dd y_*\;.
 \end{align*}

(2) \underline{Regularization in $v$.}
\smallskip
 
Similarly as above, for $\beta>0$ we regularise by convolution with $M_\beta$ in the $v$ and $(v,v_*)$ variables respectively and put
 \begin{align*}
 \mu_t^{\alpha,\beta}:=\mu^\alpha_t *_v M_\beta  \quad\text{and}\quad
\cU^{\alpha,\beta}_t=\cU^\alpha_t *_{(v,v_*)}M_\beta\;.
 \end{align*}

We denote by $f^{\alpha,\beta}_t$ and $U^{\alpha,\beta}_t$ the densities of $\mu^{\alpha,\beta}_t$ and $\cU^{\alpha,\beta}_t$ w.r.t. $\cL^d$ and Hausdorff measure on $\Omega$ respectively.
\smallskip

(3) \underline{Lower bounds.}
\smallskip

We let
\begin{align*}
g(t,x,v)=Z^{-1}\big(\jb{x-tv}+\jb{v}\big)^{-a}\;,
\end{align*}
where $a>2d+2+\mu^+$ and $Z$ is a normalisation constant (independent of $t$) such that $\|g(t,\cdot)\|_{L^1(\R^{2d})}=1$. Note that $g$ solves the transport equation
    \begin{equation}\label{eq:g-transport}
        \d_t g+v\cdot \nabla_x g=0\;,
    \end{equation}
and has the lower bound
\begin{equation*}
    \label{CR:g}
    g(t,x,v)\ge C(T)\big(\jb{x}+\jb{v}\big)^{-a}\qquad \forall x,v\in \R^d, \quad t\in [0,T]\;.
\end{equation*}

Then for $\gamma\in (0,1)$ we put $\mu^{\alpha,\beta,\gamma}_t=f^{\alpha,\beta,\gamma}_t\cL^d$ and $\cU^{\alpha,\beta,\gamma}_t=U^{\alpha,\beta,\gamma}_t\sigma$ with 
\begin{align*}
f^{\alpha,\beta,\gamma}_t:= (1-\gamma)f^{\alpha,\beta}_t +\gamma g\;,\qquad U^{\alpha,\beta,\gamma}_t = (1-\gamma)U^{\alpha,\beta}\;.
\end{align*}

(4) \underline{Regularization in $t$.}
\smallskip

Let $\eta\in C^\infty_c(\R)$ with support in $[-1,1]$ and such that $\|\eta\|_{L^1(\R)}=1$ and $\eta\geq 0$. For $\delta>0$ put $\eta_\delta(t)=\delta^{-1}\eta(t/\delta)$. 
We regularize $\mu^{\alpha,\beta,\gamma}_t$ and $\cU^{\alpha,\beta,\gamma}_t$ in time by convolution with $\eta_\delta$, i.e. we put for $t\in [0,1]$
   \begin{align}
     \mu^{\alpha,\beta,\gamma,\delta}_t:=\int \mu^{\alpha,\beta,\gamma}_{t-s}\eta_{\delta}(s)\dd s\;,\quad 
     \cU^{\alpha,\beta,\gamma,\delta}_t:=(1-\gamma)\int \cU^{\alpha,\beta}_{t-s}\eta_{\delta}(s)\dd s\;.
\end{align}
To this end, we trivially extend $\mu^{\alpha,\beta,\gamma}$ to the intervall $[-\delta,1+\delta]$ by its value at $t=0$ on $[-\delta,0]$ and its value at $t=1$ on $[1,1+\delta]$. Similarly, $\cU^{\alpha,\beta}$ is extended by the zero measure. We denote by $f^{\alpha,\beta,\gamma,\delta}_t$ and $U^{\alpha,\beta,\gamma,\delta}$ the corresponding densities.
\smallskip

We recall from Section \ref{sec:prelim} that the Gaussian convolution in the velocity variable commutes with the change of variables $T_\omega$ between pre- and post-collisional velocities, i.e.~we have $(T_\omega F)*_VM_\beta= T_\omega(F*_VM_\beta)$. Moreover, we trivially have $(T_\omega F)*_XM_\alpha= T_\omega(F*_XM_\alpha)$.

\begin{lemma}\label{lem:reg-TCRE}
The pair of regularised regularised densities $(f^{\alpha,\beta,\gamma,\delta},U^{\alpha,\beta,\gamma,\delta})$ satisfies 
\begin{equation}
\label{eq:tcre-err}
\begin{aligned}
    &\d_t \mu^{\alpha,\beta,\gamma,\delta}_t+v\cdot \nabla_x \mu^{\alpha,\beta,\gamma,\delta}_t+\mathsf E^{\alpha,\beta,\gamma,\delta}_t+\overline\nabla\cdot \cU^{\alpha,\beta,\gamma,\delta}=0\;,
\end{aligned}
\end{equation}
where the measure $\mathsf E^{\alpha,\beta,\gamma,\delta}_t$ is given by $(1-\gamma)\int \mathsf E^{\alpha,\beta}_s\eta_\delta(t-s)\dd s$, with
\[
\dd\mathsf E^{\alpha,\beta}_t(x,v)=e^{\alpha,\beta}_t(x,v)\dd x\dd v\;,\quad e^{\alpha,\beta}_t(x,v) = 2\beta\int \nabla M_\alpha(x-y)\cdot \nabla M_\beta(v-u)\dd\mu_t(y,u)\;.
\]

\end{lemma}

\begin{proof}
Starting from the fact that $(\mu,\cU)\in \TCRE$ and \eqref{eq:TCRE} a routine computation shows that also $(\mu^\alpha,\cU^\alpha)\in \TCRE$. Now, fix a test function $\phi\in C_c^\infty(\R^{2d})$. Then we have
\begin{align*}
    \frac{\dd}{\dd t} \int \phi \dd\mu^{\alpha,\beta}_t &= \frac{\dd}{\dd t} \int (\phi*_vM_\beta)\dd\mu^\alpha_t 
    = \int v\cdot\nabla_x(\phi*_vM_\beta)\dd\mu^\alpha_t +
     \frac14\int \overline\nabla (\phi*_vM_\beta)\dd\cU^\alpha_t\;.
\end{align*}
Setting $\Phi(X,V):=\phi(x,v)+\phi(\xs,\vs)$ and using
  \eqref{eq:commutation-scaling-convolution}, we find for the second term
\begin{align*}    
 \int \overline\nabla (\phi*_vM_\beta)\dd\cU^\alpha_t                     &= \int (\Phi*_VM_\beta)(X,T_\omega V)-(\Phi*_V M_\beta)(X,V) \dd\cU^\alpha_t(X,V,\omega)\\
                                &= \int ((T_\omega\Phi)*_VM_\beta)(X,V)-(\Phi*_VM_\beta)(X,V) \dd\cU^\alpha_t(X,V,\omega)\\
                                &= \int \Phi(X,T_\omega V)-\Phi(X,V)\dd(\cU^{\alpha}_t*_VM_\beta)(X,V\omega)
                                 = \int\overline\nabla\phi\dd\cU^{\alpha,\beta}_t\;.
\end{align*}
For the first term, we have, 
\begin{align*}
\int v\cdot\nabla_x(\phi*_vM_\beta)\dd\mu^\alpha_t 
&=\int v\cdot \nabla_x \phi(x,v-u) M_\beta(u)\dd\mu^\alpha_t(x,v)\dd u\\
&=\int \big((v-u)+u\big)\cdot \nabla_x \phi(x,v-u) M_\beta(u)\dd \mu^\alpha_t(x,v)\dd u\\
&= \int v\cdot \nabla_x\phi \dd\mu^{\alpha,\beta}_t
-2\beta \int \nabla_x\phi(x,v-u)\cdot \nabla M_\beta(u)\dd\mu^\alpha(x,v)\dd u\\
&= \int v\cdot \nabla_x\phi \dd\mu^{\alpha,\beta}_t
-2\beta \int \nabla_x\phi\cdot\dd(\nabla M_\beta*_v\mu^\alpha_t)\;.
\end{align*}
This yields that in distribution sense we have
\[
\partial_t\mu^{\alpha,\beta}_t + v\cdot \nabla_x\mu^{\alpha,\beta}_t +E^{\alpha,\beta}_t +\overline\nabla\cdot \cU^{\alpha,\beta}_t=0\;.
\]
Now, the claim follows by first taking \eqref{eq:g-transport} into account to obtain an analogous equation for $\mu^{\alpha,\beta,\gamma}$ and then convoluting in time with $\eta_\delta$.  
\end{proof}
\medskip

\textbf{Step 2: Estimates on the regularised curve}

Below we allow the parameters $\alpha,\beta,\gamma,\delta$ to be $0$ in the sense that the corresponding regularisation steps are skipped. Note that the regularisations by convolution in $x$, $v$, and $t$ commute.
\smallskip

\underline{ Estimates on $\mu^{\alpha,\beta,\delta,\gamma}$.}
\smallskip

We note that $(\mu^{\alpha,\beta,\gamma,\delta}_t)_t\subset\cP_{2,2+\mu^+}(\R^{2d})$ and 
\begin{equation}\label{re:bdds}
\int_0^T \cE_{2,2+\mu^+}(\mu_t^{\alpha,\beta,\gamma,\delta})\dd t \leq E\;,
\end{equation}
for a constant $E$ independent of $\alpha,\beta,\gamma,\delta$.

 If $\alpha,\beta>0$, the function $\f$ is a probability density and smooth in $x,v$ (and also in $t$ if $\delta>0$) and satisfies the pointwise bounds
\begin{align}
0< \f
&\le C(\alpha,\beta)\;,
|\log \f|&\le C(\alpha,\beta)(\jb{x}^2+\jb{v}^2)\;,\label{log:f:d:a}\\
|\nabla_x\log\f| &\leq C(\alpha,\beta)(\jb{x}+\jb{v})\;.\label{nabla:log:f}
\end{align}
For $\gamma,\delta=0$ the bound \eqref{log:f:d:a} is established in \cite[Lem.~2.6]{CC92}. The bound \eqref{nabla:log:f} is a consequence of Hamilton's gradient estimate for a solution $0\leq u(t,x)\leq M$ to the heat equation:
\[t |\nabla \log(u)|^2 \leq \log(M/u)\;\]
in combination with \eqref{log:f:d:a}. These bounds are easily checked to be preserved under the additional regularisation if $\gamma,\delta>0$. If $\gamma>0$, we trivially have in addition
\begin{equation}
     \label{bdd:log:f}
|\log\f|\le C({\alpha,\beta,\gamma})\big(1+\log\big(\jb{x}+\jb{v}\big)\big)\;.
\end{equation}

\underline{Estimates and convergence of dual dissipation potentials.}

\begin{proposition}\label{prop:potential-conv}
Let $\mu\in \cP(\R^{2d})$ and $\cU\in \cM(\Omega)$ be such that $\cR(\mu,\cU)<\infty$ and $D_{\Psi*}(\mu)<\infty$.  Then we have that
\begin{equation}\label{eq:potential-conv}
\begin{split}
&\lim_{\alpha\to 0}\cR(\mu*_{x}G_\alpha,\cU*_{(x,\xs)}G_\alpha) = \lim_{\beta\to 0}\cR(\mu*_{v}M_\beta,\cU*_{(v,\vs)}M_\beta) =\cR(\mu,\cU)\;,\\
&\lim_{\alpha\to 0}D_{\Psi*}(\mu*_{x}G_\alpha) = \lim_{\beta\to 0}D_{\Psi*}(\mu*_{v}M_\beta) =D_{\Psi*}(\mu)\;.
\end{split}
\end{equation}
Moreover, there exists a constant $C>0$ such that for all $\alpha,\beta\in(0,1)$ we have
\begin{equation}\label{eq:potential-bd}
\begin{split}
\cR(\mu*_{x}G_\alpha,\cU*_{(x,\xs)}G_\alpha), \cR(\mu*_{v}M_\beta,\cU*_{(v,\vs)}M_\beta) &\leq C \cR(\mu,\cU)\;,\\
D_{\Psi*}(\mu*_{x}G_\alpha), D_{\Psi*}(\mu*_{v}M_\beta) &\leq C D_{\Psi*}(\mu)\;.
\end{split}
\end{equation}
\end{proposition}

\begin{proof}
We follow closely the argument of \cite[Lem.~4.2]{Erb23}. We first consider convolution in velocity $v$ and the functional $D_{\Psi_*}$.  Write $\mu=f\dd v$ and  $\cU=U\dd X\dd V \dd\omega$ with $X=(x,\xs)$, $V=(v,\vs)$ and put $F(X,V)=ff_*$. Similarly, write $f^\beta=f*_v M_\beta$,
  $F^\beta=F*_VM_\beta$, $U^\beta=M_\beta*U$ for the densities of
  $\mu*_vM_\beta$, $(\mu*_vM_\beta)^{\otimes 2}$ and $\cU*_VM_\beta$. Now,
  \begin{align*}
    D_{\Psi^*}(\mu*_vM_\beta)=\int Bk\cdot G_{\Psi^*}(F^\beta,T_\omega F^\beta) \dd X\dd V\dd\omega
    =:\int L^\beta_1\dd X\dd\omega\;.
  \end{align*}
 Note that $L^\beta_1(X,V,\omega)$ converges point-wise to
  $L(X,V,\omega):=Bk\cdot G_{\Psi^*}(F,T_\omega F)$ as $\beta\to0$ and
  $\int L\dd X\dd V\dd\omega=D_{\Psi^*}(\mu)$. From the commutation property \eqref{eq:commutation-scaling-convolution}, then convexity of $G_{\Psi^*}$ and Jensen's inequality we infer that
\begin{align*}
        L^\beta_1\leq Bk\cdot\big(\big[G(F,T_\omega F)\big]*_VM_\beta\big)=:L^\beta_2\;.
      \end{align*}
      Obviously also $L^\beta_2\to L$ point-wise as $\beta\to0$. To
      prove the convergence \eqref{eq:potential-conv} it suffices by the
      extended dominated convergence theorem to show that
      $\int L^\beta_2\dd X\dd V\dd\omega\to \int L\dd X\dd V\dd\omega$. But self-adjointness of the convolution yields
      \begin{align*}
        \int L^\beta_2\dd X\dd V\dd\omega= \int \big(B*_VM_\beta\big)k\cdot G_{\Psi^*}(F,T_\omega F)\dd X\dd V\dd\omega:=\int L^\beta_3\dd X\dd V\dd\omega\;,
      \end{align*}
      and again $L^\beta_3\to L$. Now, Assumption \ref{CK} and Lemma \ref{lem:kernel-convolve} ensure the bound $B*_VM_\beta\leq C B$ and hence $L^\beta_3\leq C\cdot  L$ for a constant $C>0$. Then, dominated
      convergence gives that indeed
      $\int L^\beta_2\dd X\dd V\dd\omega=\int L^\beta_3\dd X\dd V\dd\omega \to
      \int L\dd X\dd V\dd\omega$ as desired. Note that this also provides the bound for $D_{\Psi^*}$ in \eqref{eq:potential-bd}.

The corresponding claims for the functional $\cR$ under convolution in $v$ can be proven similarly by writing 
      \begin{align*}
        \cR(\mu*_vM_\beta,\cU*_VM_\beta) = \int (Bk)^{-1}\alpha\big(U^\beta, F^\beta, T_\omega F^\beta\big)\dd X\dd V\dd\omega\;,
      \end{align*}
and using convexity of the function $G_\Psi$ from \eqref{eq:alpha} as well as the bound $B^{-1}*_VM_\beta\leq C B^{-1}$. 

Finally, statements \eqref{eq:potential-conv} and \eqref{eq:potential-bd} for convolution in $x$ are proven verbatim, recalling that convolution in $x$ trivially commutes with the change of variables $T_\omega$ and again Lemma \ref{lem:kernel-convolve}.
\end{proof}
\smallskip

\textbf{Step 3: Chain rule for the regularised curve}

We claim 
\begin{equation}
\label{chain:rule:1}
\begin{aligned}
    &\frac{\dd}{\dd r}\cH(\mu^{\alpha,\beta,\gamma,\delta})\\
    =&\int(1+\log\fr) \d_r\fr\dd x\dd v\\
=&-\int\log\fr \dd\mathsf E_r^{\alpha,\beta,\gamma,\delta}
+\frac14\int\bar\nabla \log\fr \Ur\dd\sigma.
\end{aligned}
    \end{equation}

Indeed, to show the first equality in \eqref{chain:rule:1}, we deduce from convexity of  $s\mapsto s\log s$ and the bound \eqref{bdd:log:f} the following pointwise bounds: for any $h>0$ we have
\begin{equation*}
\label{fr:t}
\begin{aligned}
  &\frac{1}{h}|f^{\alpha,\beta,\gamma,\delta}_{t+h}\log f^{\alpha,\beta,\gamma,\delta}_{t+h}-\f\log \f|\\
  &\le \frac{1}{h}\max\{\log f^{\alpha,\beta,\gamma,\delta}_{t+h},\log \f\}|f^{\alpha,\beta,\gamma,\delta}_{t+h}-\f|\\
     &\le C(\alpha,\beta, \gamma)\log\big(\jb{x}+\jb{v}\big)\|\eta'_\gamma\|_{\infty}\int f^{\alpha,\beta,\gamma}_{r}\dd r\;.
\end{aligned}
\end{equation*}
As a consequence of the integral bounds \eqref{re:bdds}, the majorant function 
is integrable on $\R^{2d}$ and we can take the time derivative inside the integral by dominated convergence.

To obtain the second equality in \eqref{chain:rule:1}, we use Lemma \ref{lem:reg-TCRE} and claim that \eqref{eq:tcre-err} can be tested against $\log f^{\alpha,\beta,\gamma,\delta}$. Indeed, we argue as in Remark \ref{rem:generaltest}.  We deduce from Corollary \ref{cor:special h} and the bounds \eqref{bdd:log:f},\eqref{nabla:log:f}, as well as \eqref{eq:potential-bd} and \eqref{re:bdds}, that \eqref{eq:finite-test} holds for $\cU^{\alpha,\beta,\gamma,\delta}$ with $\varphi=\log f^{\alpha,\beta,\gamma,\delta}$. This allows to pass to the limit when approximating $\log\f$ with test functions in the transport term and the term involving $\cU^{\alpha,\beta,\gamma,\delta}$. Also the term involving $\mathsf E^{\alpha,\beta,\gamma,\delta}$ passes to the limit thanks to the bounds  \eqref{bdd:log:f} and \eqref{re:bdds} and we obtain
\begin{align*}
  \frac{\dd}{\dd r}\cH(\mu^{\alpha,\beta,\gamma,\delta})
  =&
  -\int \log\fr v\cdot\nabla_x\fr\dd x\dd v\\
  &+\int\log\fr \dd\mathsf E_r^{\alpha,\beta,\gamma,\delta}
+\frac14\int\overline\nabla \log\fr \dd\cU^{\alpha,\beta,\gamma,\delta}_r\;.
\end{align*}
We note, using integration by parts and the decay of $\fr$ at infinity, that the first term on the right-hand side vanishes.

Finally, we integrate \eqref{chain:rule:1} over $[s,t]$ to obtain
\begin{equation}
\label{chain:rule:2}
\begin{aligned}
&\cH(\mu^{\alpha,\beta,\gamma,\delta}_t)-\cH(\mu^{\alpha,\beta,\gamma,\delta}_s)\\
&=\int_s^t\int\log\fr\dd \mathsf E^{\alpha,\beta,\gamma,\delta}_r\dd r+\frac14\int_s^t\int\overline\nabla \log\fr \dd\cU^{\alpha,\beta,\gamma,\delta}_r\dd r\;.
\end{aligned}
\end{equation}

\medskip

\textbf{Step 4: Passing to the limit.}

We will pass to the limit in \eqref{chain:rule:2} by letting $\delta,\gamma,\beta,\alpha\to0$ in this order. We discuss each term individually. We will repeatedly 
use the dominated convergence theorem in the following form (see e.g.~\cite[Chap.~4,Thm.~17]{Roy88}): Let $(R_\delta)_{\delta>0}$ and $(I_\delta)_{\delta>0}$ be
families of measurable functions on a measure space $X$ with
$I_\delta\geq0$ and let $R,I$ be measurable. Assume that
$R^\delta,I^\delta$ converge point-wise to $R,I$ respectively,
$|R^\delta|\leq I^\delta$ a.e., and
$\lim_{\delta\to0}\int_XI^\delta=\int_XI$. Then we also have
$\lim_{\delta\to0}\int_X R^\delta=\int_XR$.

\textbf{RHS: Collision term.}
Consider the term
\[
\int_s^t\int \overline\nabla\log \fr U^{\alpha,\beta,\gamma,\delta}_r\dd\sigma\dd r\;.
\]

\begin{enumerate}
    \item \underline{Limit $\delta\to0$.}
    
The bound $|\log\fr|\leq C\big(1+\log( \jb{x}+\jb{v})\big)$ which is uniform in $\delta$, together with the integrability of $U_r^{\alpha,\beta,\gamma,\delta}$ given by Corollary \ref{cor:special h} and the uniform moment bound \eqref{re:bdds}, allow us to pass to the limit $\delta\to0$ by dominated convergence
to obtain
\begin{align}\label{eq:abc}
\int_s^t\int \overline\nabla\log f^{\alpha,\beta,\gamma}_r U^{\alpha,\beta,\gamma}_r\dd\sigma\dd r\;.
\end{align}

\item \underline{Limit $\gamma\to0$.}

We drop $\alpha$, $\beta$, and $r$ from the notation for convenience. Then, as in Remark \ref{rem:psi-psi*-id} with \eqref{eq:dual-est} and \eqref{eq:bd-psi-star}, as well with the bound \eqref{log:f:d:a}, we get the following estimate:
\begin{align*}
&|\overline\nabla \log f^\gamma \cdot U^\gamma|
=(1-\gamma)^{-1}|\overline\nabla \log f^\gamma \cdot (1-\gamma)U^\gamma|\\
&\leq
(1-\gamma)^{-1}\Psi\Big(\frac{(1-\gamma)U^\gamma}{\theta(f^\gamma)Bk}\Big) \theta( f^\gamma)Bk
+ |\overline\nabla  \log f^\gamma | \big|  f^\gamma (f^\gamma)_* - (f^\gamma)'( f^\gamma )'_*  \big|Bk\\
&\leq 
(1-\gamma)\Psi\Big(\frac{U}{\theta(f)Bk }\Big) \theta(f)Bk
+ C(|\xi|+|\xi_*|) \big(  f^\gamma (f^\gamma)_*)+(f^\gamma)'( f^\gamma )'_*  \big)Bk\;.
\end{align*}
Here we have set $\xi(x,v)=\jb{x}^2+\jb{v}^2$ and used that $U^\gamma=(1-\gamma)U$, $f^\gamma=(1-\gamma)f+\gamma g\geq (1-\gamma)f$ as well as the monotonicity of $\theta$ and \eqref{eq:psi-inc}. Note that the integral of the first term in the last line gives rise to $\int\cR(\mu^{\alpha,\beta}_r,\cU^{\alpha,\beta}_r)\dd r$.  By our Assumption \ref{CK} and the moment bounds \eqref{re:bdds} the integral on  $[0,T]\times\Omega$ of the second term converges as $\gamma\to0$. Hence we can pass to the limit in \eqref{eq:abc} as $\gamma\to0$ by the dominated convergence theorem and obtain.
\begin{align}\label{eq:ab}
\int_s^t\int \overline\nabla\log f^{\alpha,\beta}_r U^{\alpha,\beta}_r\dd\sigma\dd r\;.
\end{align}

\item \underline{Limits $\beta,\alpha\to0$.}
 
 Let us discuss the limit $\beta\to0$ and drop $\alpha$ and $r$ from the notation. Note that $\overline\nabla\log f^{\beta}U^{\beta}$ converges point-wise
  to $\overline\nabla\log fU$ as $\beta\to0$ at every $r$ where the
  densities of $\mu_r,\cU_r$ exist. We use again \eqref{eq:dual-est} to obtain the following majorant:
   \begin{align*}
    |\overline\nabla \log f^{\beta}U^{\beta}|
    &\leq    
   \Psi\Big(\frac{U^{\beta}}{\theta(f^{\beta})Bk}\Big)\theta(f^{\beta})Bk + \Psi^*\Big(\overline\nabla\log f^{\beta}\Big)\theta(f^{\beta})Bk\\
     &= G_{\Psi}\Big(F^{\beta}Bk,(F^{\beta})'Bk,U^{\beta}\Big) + G_{\Psi^*}\Big(F^{\beta}Bk,(F^{\beta})'Bk\Big)\\
    &=:
I_1^\beta+I_2^\beta\;,
  \end{align*}
  where we have set $F^{\beta}=f^{\beta}f^\beta_*$, $(F^\beta)'=T_\omega F^\beta$ and recall \eqref{eq:alpha}.
  Obviously $I_1^{\beta}\to I^0_1$ and $I_2^{\beta}\to I^0_2$ point-wise,
  where $I_1^0$ and $I_2^0$ are the corresponding expressions with
  $f^{\beta}$ and $U^{\beta}$ replaced by $f, U$. Note that the integral over $\Omega$ of $I^\beta_1,I^\beta_2$ give rise to $\cR(\mu^\beta,\cU^\beta)$ and $D_{\Psi^*}(\mu^\beta)$. By Proposition \ref{prop:potential-conv} these converge to $\cR(\mu,\cU)$ and $D_{\Psi^*}(\mu)$ for a.e.~$r\in [0,T]$. Thus by the dominated convergence theorem we can pass to the limit in the integral over $\Omega$ in \eqref{eq:ab}. Finally, to pass to the limit in the time integral, we use the already established almost everywhere in time convergence of the $\Omega$ integral and exhibit a majorant similar as above using \eqref{eq:potential-bd}. The limit $\alpha\to0$ can be discussed verbatim.
\end{enumerate}
    
\textbf{RHS: Error term.}

We discuss the term
\begin{align*}
&\int_s^t\int\log\fr \dd\mathsf E_r^{\alpha,\beta,\gamma,\delta}\dd r\\
&=
-(1-\gamma)\int_s^t\int\nabla_x\log\fr(x,v)\cdot u M_\beta(u)\eta_\delta(a)f^\alpha_{r-a}(x,v-u)\dd(a,u,x,v,r)\;.
\end{align*}
(For notational convenience, we assume that $\mu_r$ has density $f_r$ and write $f^\alpha=M_\alpha*_x f$, $f^\beta=M_\beta*_v f$.)
We claim that this term vanishes in the limit $\delta,\gamma,\beta\to0$, taken in this order.

The bound \eqref{nabla:log:f} togehter with the moment bound \eqref{re:bdds} allows us to pass to the limit $\delta\to0$ and subsequently $\gamma\to0$ by dominated convergence. Thus we are left with
\begin{align*}
&-\int_s^t\int\nabla_x\log f_r^{\alpha,\beta}(x,v)\cdot uM_\beta(u)f^\alpha_{r}(x,v-u)\dd u \dd x\dd v\dd r\\
&=
-\int_s^t\int\frac{\nabla_x f_r^{\alpha,\beta}}{ f_r^{\alpha,\beta}}(x,v)\cdot uM_\beta(u)f^\alpha_{r}(x,v-u)\dd u \dd x\dd v\dd r\;.
\end{align*}

By using Cauchy--Schwarz inequality, the last term can be bounded by
\begin{align}\label{eq:AB}
\Big[\int_s^t\int\frac{|\nabla_xf^{\alpha,\beta}_r|^2}{f^{\alpha,\beta}_r}\dd r\Big]^{\frac12}\times\Big[\int_s^t\int\frac{|\int |u|f^{\alpha}_r(x,v-u)M_\beta(u)\dd u|^2}{\int f^{\alpha}_r(x,v-u)M_\beta(u)\dd u}\dd x\dd v\dd r\Big]^{\frac12}\;.
\end{align}
For the first term in \eqref{eq:AB} we estimate by convexity of $(a,b)\mapsto a^2/b$ and Jensen's inequality:
\begin{align*}
\int\frac{|\nabla_xf^{\alpha,\beta}_r|^2}{f^{\alpha,\beta}_r}&=\int\frac{\Big|\int\nabla M_\alpha(y)f^\beta_r(x-y,v)\dd y\Big|^2}{\int M_\alpha(y)f^\beta_r(x-y,v)\dd y}\dd x\dd v
=\alpha^{-2}\int\frac{\Big|\int yM_\alpha(y)f^\beta_r(x-y,v)\dd y\Big|^2}{\int M_\alpha(y)f^\beta_r(x-y,v)\dd y}\dd x\dd v\\
&\le \alpha^{-2}\iint|y|^2M_\alpha(y)f^{\beta}_r(x-y,v)\dd y\dd x\dd v=\alpha^{-1}\;.
\end{align*}
Similarly, we obtain for the second term:
\begin{align*}
&\int\frac{\big|\int|u|f^{\alpha}_r(x,v-u)M_\beta(u)\dd u\big|^2}{\int f^{\alpha}_r(x,v-u)M_\beta(u)\dd u}\dd x\dd v
\leq\iint |u|^2M_\beta(u)f^{\alpha}_r(x,v-u)\dd u\dd x\dd \dd v
=\beta\;.
\end{align*}
Thus, the expression in \eqref{eq:AB} is bounded by $(t-s)\sqrt{\beta/\alpha}$ and hence the error term vanishes as $\beta\to0$ while $\alpha>0$ is fixed.
\medskip

\textbf{LHS: Entropy terms.}
 
Finiteness of the second moment of $\mu_r$ for each $r$ and the uniform moment bounds in \eqref{re:bdds} imply that for each $r$, the entropy $\cH(\mu^{\alpha,\beta,\delta,\gamma}_r)$ is bounded away from $-\infty$ uniformly in the regularisation. By convexity of $s\mapsto s\log s$ and the bound \eqref{bdd:log:f}, we have the estimate 
\begin{align*}
  &|\cH(f^{\alpha,\beta,\gamma,\delta}_r)-\cH(f_r^{\alpha,\beta,\gamma}))|\\
&\le
 C\int_{-1}^1\int_{\R^{2d}}(1+\log(\jb{x}+\jb{v})|f^{\alpha,\beta,\gamma}_{r-\delta s}-f_r^{\alpha,\beta,\gamma}|\eta(s)\dd x\dd \dd v\dd s
 \end{align*}
for all $r\in[0,T]$. By Remark \ref{rem:generaltest} and Corollary \ref{cor:special h} have that $r\mapsto \int (1+\log(\jb{x}+\jb{v})\dd\mu_r(x,v)$ is continuous (the additional term error term in the transport collision rate equation does not pose a problem), and hence for all $r$ the term above vanishes as $\delta\to0$.

The bound \eqref{log:f:d:a} yields the convergence $\cH(\mu^{\alpha,\beta,\gamma}_r)\to\cH(\mu^{\alpha,\beta}_r)$ as $\gamma\to0$ for all $r$. Thus we are left with $\cH(\mu_t^{\alpha,\beta})-\cH(\mu_s^{\alpha,\beta})$ in the left hand side. By convexity of $s\mapsto s\log s$ we have $\cH(\mu^{\alpha,\beta}_r)\leq \cH(\mu_r)$. We have $\mu^{\alpha,\beta}_r\to \mu_r$ weakly as $\alpha,\beta\to0$ and the second moments converge and hence by the lower semicontinuity of $\cH$ we deduce that $\cH(\mu^{\alpha,\beta}_r)$ increases to $\cH(\mu_r)$ for all $r$. Hence $\cH(\mu^{\alpha,\beta}_t)-\cH(\mu^{\alpha,\beta}_s)$ converges to $\cH(\mu_t)-\cH(\mu_s)$ for all $s,t$ and the finiteness of $\cH(\mu_0)$ and the boundedness of the right hand side in the limit $\alpha,\beta\to0$ yields that $\cH(\mu_t)$ is finite for all $t$.

\subsection{Variational characterisation}
\label{sec:var-char}

Now we are ready to prove our main theorem.
\begin{theorem}\label{thm:main}
Let $(\mu,\cU)\in \TCRE_T$ with $(\mu_t)_t\subset\cP_{2,2+\mu_+}(\R^{2d})$ and  $\cH(\mu_0)<\infty$ such that
\[
\int_0^T\int_{\R^{2d}}|x|^2+|v|^{2+\mu_+}\dd \mu_t\dd t<\infty\;,\]
where $\mu_+=\max(0,\mu)$. Then we have that
\begin{equation*}\label{eq:LT}
\cL_T(\mu,\cU):=\cH(\mu_T)-\cH(\mu_0) +\int_0^TD_{\Psi^*}(\mu_t)+\cR(\mu_t,\cU_t)\dd t \ge 0\;.
\end{equation*}
Moreover, we have $\cL_T(\mu,\cU)=0$ if and only if $\mu_t$ has density $f_t$ with $(f_t)$ a weak solution of \eqref{FBE} with  
\begin{equation}
\label{ass:D}
\int_0^T\cD(f_t)\dd t<\infty\;.
\end{equation}
In this case, we have the entropy identity
\begin{equation}
\label{cH-var}
    \cH(f_t)-\cH(f_0)=-\int_0^t \cD(f_s)\dd s\quad\forall t\in[0,T].
\end{equation}
\end{theorem}

\begin{remark}
Recall that by Theorem \ref{thm:1} for initial data $f_0\in L_{2,2+\mu^+}(\R^{2d})$ with $\cH(f_0)<\infty$, there exists a unique weak solution and that this solution satisfies \eqref{ass:D}.     
\end{remark}

\begin{proof}
Let $(\mu,\cU)\in \TCRE_T$ satisfy the assumptions of the theorem. To show that $\cL_T(\mu,\cU)\geq 0$, we can assume that \eqref{eq:chain-ass} holds. Otherwise the estimate would hold trivially recalling that finiteness of the second moment implies that $\cH(\mu_t)>-\infty$ for every $t\in[0,T]$. Now, the chain rule Theorem \ref{thm:chainrule} yields
\[\cH(\mu_T)-\cH(\mu_0)=\int_0^T\int_{\{\theta(f_t)>0\}} \overline\nabla \log f_t\dd\cU_t\dd t\;,\]
where $f_t$ denotes the density of $\mu_t$. Let $\cU=U\sigma$. Then we deduce then inequality $\cL_T(\mu,\cU)\geq 0$ by using the estimate  \eqref{eq:dual-est} on the set $\{\theta(f)>0\}$ to obtain:
\begin{align}\nonumber
\frac14\overline\nabla \log f \cdot U &= -\frac14 \theta(f)(-\overline\nabla \log f)Bk\frac{U}{\theta(f)Bk}\\\label{eq:chain-est-pw}
&\geq - \frac14 \Psi\big(\frac{U}{\theta(f)Bk}\big)\theta(f)Bk -\frac14\Psi^{*}\big(-\bar\nabla \log f\big)\theta(f)Bk\;.
\end{align}
Now, assume that $\cL_T(\mu,\cU)=0$. Then for a.e.~$t$ and a.e.~on the set
$\{\theta(f_t)>0\}$ we must have equality in \eqref{eq:chain-est-pw} and hence by Remark \ref{rem:psi-psi*-id}
\[U = (\Psi^{*})'(-\bar\nabla \log f)\theta(f)Bk=
   \big[f\fS-\fP\fSP\big]Bk\;.\] 
On the set where $\theta(f_t)=0$ and hence $U_t=0$ we must have
$G_{\Psi^*}(f\fS,\fP\fSP)=0$. But the assumptions on $G_{\Psi^*}$ and
$\theta$ then yield that on $\{\theta(f)=0\}$ we have $f\fS=\fP\fSP=0$
and thus again $U=f\fS-\fP\fSP$, i.e.~$(f_t)$ is a weak solution to the fuzzy Boltzmann equation. Moreover, from $U=f\fS-\fP\fSP$ we immediately obtain that 
$D_{\Psi*}(\mu)+\cR(\mu,U)=\cD(\mu)$ and hence \eqref{cH-var} holds. Since $\cH(\mu_T)>-\infty$, \eqref{ass:D} holds in particular.

Conversely, if $(f_t)$ is a weak solution to the fuzzy
Boltzmann equation satisfying \eqref{ass:D}, then we have $(\mu,\cU)\in \TCRE_T$ for $\cU=U\sigma$ with $U=f\fS-\fP\fSP$ and we deduce along the previous lines that $\cL_T(\mu)=0$.
\end{proof}

\appendix  

\section{GENERIC formulation}
\label{app:GS}

Here we spell out in more detail how the fuzzy Boltzmann equation \eqref{FBE} can be cast formally into the GENERIC framework.

In the following, we choose as state space $\mathsf Z$ the set of probability densities on $\R^{2d}$ that are smooth enough and with appropriate decay so that all the manipulations below are justified. As tangent and cotangent space at $f\in \mathsf Z$ we formally consider the sets $T_f\mathsf Z=\{\xi\in L^2(\R^{2d}) : \int \xi=0\}$ and $T^*_f\mathsf Z=L^2(\R^{2d})$. As their duality pairing we consider the $L^2$ inner product, i.e. $\xi\cdot \eta=\int \xi \eta$.

We define the total energy $\mathsf E:\mathsf Z\to\R$ and entropy $\mathsf S:\mathsf Z\to\R$ via
\begin{align*}
\mathsf E(f)&=\frac12\int_{\R^{2d}}|v|^2f\dd v\dd x\;,\qquad \mathsf S(f)=-\cH(f)\;.
\end{align*}
For each $f\in \mathsf Z$ we define operators $\mathsf M(f)$ and $\mathsf L(f)$ acting on sufficiently smooth and decaying function $g$ via 
\begin{align*}
\mathsf M(f)g&=-\frac14\overline\nabla\cdot (Bk\Lambda (f)\overline\nabla g),\\ 
\mathsf L(f)g &=-\div(f J\nabla g),\quad J=\begin{pmatrix}
0 & {\sf id}\\
-{\sf id} & 0
\end{pmatrix},
\end{align*}
where $\nabla=\begin{pmatrix}
    \nabla_x\\
    \nabla_v
\end{pmatrix}$ denotes the gradient on $\R^{2d}$. $\Lambda(f)$ is shorthand for $\Lambda(ff_*,f'f_*')$ with $\Lambda(s,t):=(s-t)/(\log s-\log t)$ denoting the logarithmic mean and finally the "divergence" $\overline\nabla\cdot$ is the adjoint in $L^2$ of $\overline\nabla$. More precisely, for a function $U:\R^{4d}\times S^{d-1}\to \R$ we set 
\begin{align*}
\overline \nabla \cdot U (x,v) := \frac14\int
\big[&U(x,x_*,v,v_*,\omega)+U(x_*,x,v_*,v,\omega)\\
-&U(x,x_*,v',v_*',\omega)-U(x_*,x,v_*',v',\omega) \big]\dd x_*\dd v_*\dd\omega\;.
\end{align*}
One readily checks, that $\mathsf M$ is symmetric and positiv semi-definite and $\mathsf L$ is antisymmetic in the sense specified in Section 1.2. Moreover, a direct computation yields the validity of the Jacobi identity for $\mathsf L$.

We compute the differential of energy and entropy (w.r.t. $L^2$ or in other words the functional derivative) as
\begin{itemize}
    \item $\mathsf d\mathsf E(f)=\frac12 |v|^2$ and
    \begin{align*}
      \mathsf L(f)\mathsf d \mathsf E(f)&= -\frac12\div(f J\nabla v^2)=-\div\big(f\begin{pmatrix}
0 & {\sf id}\\
-{\sf id} & 0
\end{pmatrix}\begin{pmatrix}
0\\v
\end{pmatrix} \big)\\
&=-\div\big(\begin{pmatrix}
vf\\0
\end{pmatrix}\big)=-v\cdot\nabla_x f\;.
    \end{align*}
    \item $\mathsf d \mathsf S(f)=-(\log f+1)$ and
    \begin{align*}
       \mathsf M(f)\mathsf d\mathsf S(f)&=\frac14\overline\nabla\cdot \big(Bk\Lambda(f)\overline\nabla \log f\big)\dd x\dd v\\
&=\int Bk\big(f_*'f'-f_* f\big)\dd x_*\dd v_*\dd\omega\;.
    \end{align*}
\end{itemize}
    Hence, the corresponding GENERIC equation $\d_t f=\mathsf L\mathsf d \mathsf E + \mathsf M\mathsf d \mathsf S$ is indeed the same as the fuzzy Boltzmann equation \eqref{FBE}.

To verify that $(\mathsf Z,\mathsf E,\mathsf S,\mathsf L, \mathsf M)$ indeed constitute a GENERIC system, it remains to verify the non-interaction conditions 
\begin{equation*}
\mathsf L \mathsf d\mathsf S=\mathsf M\mathsf d\mathsf E=0\;.    
\end{equation*}
We obtain 
\begin{align*}
        \mathsf L\mathsf d \mathsf S(f)&=-\div (fJ\nabla \log f)\\
        &=-\div\big(\begin{pmatrix}
            0& {\sf id}\\
            -{\sf id} & 0
        \end{pmatrix}
        \begin{pmatrix}
            \nabla_x f\\
            \nabla_v f
        \end{pmatrix}\big)\\
        &=-\begin{pmatrix}
            \nabla_x\\
            \nabla_v
        \end{pmatrix}\cdot 
        \begin{pmatrix}
\nabla_v f\\
-\nabla_x f
        \end{pmatrix}=0
\end{align*}
and
    \begin{align*}
        \mathsf M\mathsf d \mathsf E&= -\frac18 \overline\nabla\cdot \big(Bk\Lambda(f) \overline\nabla |v|^2\big)=0\;,
    \end{align*}
where we used the fact that collision conserve energy, i.e. $|v|^2+|\vs|^2=|\vp|^2+|\vsp|^2$.
\printbibliography
\end{document}